\documentclass{article}
\usepackage{amsfonts,amsthm}
\usepackage{authblk}
\usepackage{newpxtext,newpxmath}
\usepackage{pdflscape}
\usepackage{multirow}
\usepackage{multicol}
\usepackage{tikz}
\usepackage{geometry}

\usepackage{arydshln}
\usepackage{mathrsfs}

\theoremstyle{plain}
\newtheorem{theorem}{Theorem}[section]
\newtheorem{lemma}[theorem]{Lemma}
\newtheorem{proposition}[theorem]{Proposition}
\newtheorem{corollary}[theorem]{Corollary}

\theoremstyle{definition}
\newtheorem{definition}{Definition}[section]

\newtheorem{example}{Example}[section]

\theoremstyle{remark}
\newtheorem{remark}{Remark}[section]

\newenvironment{proofB}{ \noindent \textit{Proof of the }\textbf{Theorem B}.}{\hfill$\square$}

\title{Classification of orbit closures in the variety of 4-dimensional symplectic Lie
algebras}

\makeatletter
\def\blfootnote{\xdef\@thefnmark{}\@footnotetext}
\makeatother

\makeatletter
\renewcommand\@date
{{%

\vspace{-1.5cm}%
\large\centering

  \begin{tabular}{@{}c@{}}
    Edison Alberto \textsc{Fern\'andez-Culma} \textsuperscript{1,3,$\ddagger$}\blfootnote{${}^{\ddagger}$ Partially supported by CONICET and SECyT-UNC: 33620180100523CB} \\
    \texttt{\small efernandez@famaf.unc.edu.ar}
  \end{tabular}%

  and

  \begin{tabular}{@{}c@{}}
   Nadina  \textsc{Rojas} \textsuperscript{1,2,3,$\ddagger$} \\
     \texttt{ \small nadina.rojas@unc.edu.ar}
  \end{tabular}

  \textsuperscript{1}{\small Centro de Investigaci\'on y Estudios en Matem\'atica de C\'ordoba, CONICET}\par
  \textsuperscript{2}{\small Facultad de Ciencias Exactas, F\'isicas y Naturales, UNC}\par
  \textsuperscript{3}{\small Facultad de Matem\'atica, Astronom\'ia, F\'isica y Computaci\'on, UNC}\par
  \textsuperscript{ }{\small C\'ordoba, Argentina}\par
}}
\makeatother

\author{}

\begin{document}

\maketitle



\begin{abstract}
The aim of this paper is to study the natural action of the real symplectic group, $\operatorname{Sp}(4, \mathbb{R})$, on the algebraic set of $4$-dimensional Lie algebras admitting symplectic structures and to give a complete classification of orbit closures.

We present some applications of such classification to the study of the Ricci curvature of left-invariant almost K\"{a}hler structures on four dimensional Lie groups.
\end{abstract}

\section{Introduction}

An almost K\"ahler manifold is a real manifold endowed with a Riemannian metric $g$, a symplectic form $\omega$ and an almost-complex structure $J$, which satisfy the \textit{compatibility condition}:
\begin{eqnarray*}
\omega(\cdot, \centerdot ) & = & g(J\cdot, \centerdot).
\end{eqnarray*}
If in addition the almost-complex structure $J$ is integrable, then $(M,g,\omega, J)$ is a K\"{a}hler manifold.

Let $(M,\omega)$ be a symplectic manifold. The set of all pairs $(g,J)$ such that $(M,g,\omega, J)$ is an almost K\"ahler manifold is usually denoted by
$\mathbf{AK}(M,\omega)$, and is called the space of all {$\omega$-compatible almost K\"{a}hler structures} (equivalently, the space of all {$\omega$-compatible metrics}), which is well known to be an infinite dimensional Frechet manifold and totally geodesic in the set of all Riemannian metrics on $M$ (see \cite{blair}).

The problem of finding distinguished Riemannian metrics in $\mathbf{AK}(M,\omega)$ has been intensively studied since 1969, when Samuel I. Goldberg proved that if the almost complex structure of an almost K\"{a}hler manifold commutes with the curvature operator, then it is integrable and conjectured that a compact almost-K\"{a}hler Einstein manifold is necessarily K\"{a}hler (\cite{Goldberg}). It is worth mentioning that Kouei Sekigawa proved that the conjecture is true if the scalar curvature is non-negative (\cite{Sekigawa}) and there are some positive partial results in dimension four under additional assumptions on the curvature (see the very nice survey \cite{ApostolovDraghici} and the references given there).

The aim of this work is to study left-invariant almost K\"{a}hler structures on four dimensional Lie groups. In this setting, the nature of the problem is inherently linear. We follow closely the point of view  developed by Jens Heber in \cite{Heber}: he varies \textit{``Lie algebra structures instead of scalar products''} to study left-invariant metrics on Lie groups. In our case, let ${\omega_{\tiny{\mbox{cn}}}} = e_1^{\ast}\wedge e_3^{\ast} + e_2^{\ast}\wedge e_4^{\ast} $  be
the \textit{canonical symplectic form} on $\mathbb{R}^{4}$ and let $\mathcal{L}_{\omega_{\tiny{\mbox{cn}}}}(\mathbb{R}^{4})$ denote the set of all the Lie brackets $\mu$
on $\mathbb{R}^{4}$ such that $(\mathbb{R}^{4}, \mu, {\omega_{\tiny{\mbox{cn}}}} )$ is a \textit{symplectic Lie algebra}. 
It is easy to show that the $\mbox{Sp}(\mathbb{R}^{4},{\omega_{\tiny{\mbox{cn}}}})$-orbit of $\mu$ under the natural action of $\mbox{Sp}(\mathbb{R}^{4},{\omega_{\tiny{\mbox{cn}}}})$ on $\mathcal{L}_{\omega_{\tiny{\mbox{cn}}}}(\mathbb{R}^{4})$ parametrizes all left-invariant metrics that are compatible with $(\mathbb{R}^{4}, \mu,{\omega_{\tiny{\mbox{cn}}}} )$.

Given a four-dimensional symplectic Lie algebra $(\mathbb{R}^{4}, \mu, {\omega_{\tiny{\mbox{cn}}}} )$, we classify all Lie algebras whose Lie bracket is in the closure of the orbit $\mbox{Sp}(\mathbb{R}^{4},{\omega_{\tiny{\mbox{cn}}}}) \cdot \mu$ (with respect to the subspace topology on $\mathcal{L}_{\omega_{\tiny{\mbox{cn}}}}(\mathbb{R}^{4})$). In other words, we classify all \textit{degenerations} of Lie algebras of dimension $4$ admitting a symplectic structure with respect to the natural action of the symplectic group $\mbox{Sp}(4,\mathbb{R})$: \textbf{Theorem A}.

Originally, the study of degenerations of linear algebraic structures, mainly in the case of associative algebras or Lie algebras, was widely stimulated by applications in Mathematical physics, Geometry and Algebra. For instance, in his seminal paper on left invariant metrics on Lie groups, John Milnor used (implicitly) degenerations of Lie algebras to prove existence the Riemannian metrics on Lie groups with certain curvature properties (see, e.g., \cite[Theorem 2.5]{Milnor}). As far as we are concerned, by using \textbf{Theorem A}, we can show that:\newline

\noindent \textbf{Theorem B} $ $ \textit{Let $G$ be a four-dimensional Lie group admitting a left-invariant symplectic structure. Then $G$ admits a left-invariant almost-K\"{a}hler structure with non-degenerate Ricci curvature (of signature $(+,-,-,-)$) if and only if the Lie algebra of $G$ is not isomorphic to the $\mathfrak{h}_{3} \times \mathbb{R}$, $\mathfrak{aff}(\mathbb{R})\times \mathbb{R}^2$ or the four dimensional abelian Lie algebra}.\newline

Here, $\mathfrak{h}_{3}$ denotes the $3$-dimensional Heisenberg Lie algebra and $\mathfrak{aff}(\mathbb{R})$ is the Lie algebra of the Lie group of affine motions of $\mathbb{R}$.

Our study is based on the classification of four-dimensional symplectic Lie algebras obtained by Gabriela Ovando in \cite{ovando}
and to reduce the length of our exposition, we use results given in \cite{nesterenko} by Maryna Nesterenko and Roman Popovych about degenerations of 4-dimensional real Lie algebras. Also, we introduce new symplectic invariants for symplectic Lie algebras which behave well under degeneration, and are easily calculable.

\section{Preliminaries}

\begin{definition}[Symplectic Lie algebras]
Let $\mathfrak{g}$ be a $2n$-dimensional real Lie algebra. A \textit{symplectic structure} on $\mathfrak{g}$ is a closed non-degenerate $2$-form $\omega \in \Lambda^{2}(\mathfrak{g})$, i.e.:

\begin{itemize}
  \item (Closed) $\operatorname{d}\!\omega\,(v_{1},v_{2},v_{3}) = \displaystyle \sum_{\sigma \in \operatorname{S}_3} \operatorname{sign}(\sigma) \, \omega( [v_{\sigma(1)}, v_{\sigma(2)}], v_{\sigma(3)}) =0 $,\newline for all $v_1,v_2,v_3 \in \mathfrak{g}$. Here $\operatorname{S}_3$ denotes the symmetric group of degree $3$.
  \item (Non-degenerate) If $\omega(v,x)=0$ for all $x \in \mathfrak{g}$, then $v$ must be the zero vector of $\mathfrak{g}$, or equivalently $\omega^{n}$ is an \textit{orientation form} for $\mathfrak{g}$ (see for instance \cite[Proposition 22.8]{Lee}).
 \end{itemize}
and the pair $(\mathfrak{g},\omega)$ is called a \textit{symplectic Lie algebra}.
\end{definition}

\begin{definition}[Equivalence of symplectic structures]
 Let $(\mathfrak{g}_{1}, \omega_{1})$ and  $(\mathfrak{g}_{2}, \omega_{2})$ be two symplectic Lie algebras. Then, they are said to be \textit{symplectomorphically equivalent} if there exists an isomorphism of Lie algebras $g:\mathfrak{g}_{1} \rightarrow \mathfrak{g}_{2}$ such that $g$ preserves the symplectic structures; i.e. the pullback $g^{\ast} \omega_{2}$ is equal to $\omega_{1}$.
\end{definition}

\begin{remark}
The notion of equivalence of symplectic structures allows to consider the problem of classifying all possible symplectic structures on a fixed Lie algebra $\mathfrak{g}$ (up to symplectic automorphisms); such set is parametrized by the moduli space $(Z^{2}(\mathfrak{g}) \cap \Omega^{2}(\mathfrak{g}))/\operatorname{Aut}(\mathfrak{g})$, where $Z^{2}(\mathfrak{g})$ is the vector space of $2$-cocycles on $\mathfrak{g}$ with values in the trivial module $\mathbb{R}$, and $\Omega^{2}(\mathfrak{g})$ is the open set of non-degenerate $2$-forms on $\mathfrak{g}$ and $\operatorname{Aut}(\mathfrak{g})$ is acting via \textit{pullback} on $\Lambda^{2}(\mathfrak{g})$ (see for instance \cite[\S 2.1.]{moscosoTamaru}).
\end{remark}

\subsection{The variety of symplectic Lie algebras}

Suppose $V$ and $W$ are finite-dimensional real vector spaces, and let $L^{k}(V;W)$ denote the vector space of all multilinear maps from $V^{k}$, the $k$-fold Cartesian product of $V$, to $W$, and let $C^{k}(V;W) \subseteq L^{k}(V;W)$ denote the subspace consisting of all multilinear maps that are \textit{alternating}; i.e. $\mu \in L^{k}(V;W)$ such that $\mu(v_{\sigma(1)}, v_{\sigma(2)}, \ldots , v_{\sigma(k)})$ is equal to $\operatorname{sign}(\sigma) \, \mu(v_1,v_2, \ldots, v_k)$ for each permutation $\sigma \in \operatorname{S}_{k}$.

The general linear group of $V$, $\operatorname{GL}(V)$, acts on $L^{k}(V;V)$ and $L^{k}(V;\mathbb{R})$ by \textit{change of basis}: for all $g\in \operatorname{GL}(V)$, $\mu \in L^{k}(V;V)$ and $\alpha \in L^{k}(V;\mathbb{R})$
\begin{eqnarray*}
  g\cdot\mu(\cdot,\ldots,\cdot) &:=& g\mu (g^{-1}\cdot , \ldots , g^{-1}\cdot), \mbox{ and }\\
  g\cdot\alpha(\cdot,\ldots,\cdot)&:=& \alpha(g^{-1}\cdot , \ldots , g^{-1}\cdot).
\end{eqnarray*}

Recall that if $(V,\omega)$ is a symplectic vector space, then there exists a basis for $V$, say $\mathcal{B} = \{e_1, \ldots, e_n, e_{n+1}, \ldots, e_{2n}\}$, such that:
$$ \omega = \sum_{i=1}^{n} e_{i}^{\ast}\wedge e_{n+i}^{\ast},$$
where $\mathcal{B}^{\ast} = \{e_1^{\ast}, \ldots, e_n^{\ast}, e_{n+1}^{\ast}, \ldots, e_{2n}^{\ast}\} $ denotes the corresponding dual basis for $V^{\ast}$ (see for instance \cite[Proposition 22.7.]{Lee}). Such a basis is called a \textit{symplectic basis} for  $(V,\omega)$.

From now on we assume that $V$ is $\mathbb{R}^{2n}$ and let us denote by $\omega_{\operatorname{cn}}$ the \textit{canonical symplectic form} on $\mathbb{R}^{2n}$, $\omega_{\operatorname{cn}}:= e_{1}^{\ast}\wedge e_{n+1}^{\ast} + \ldots + e_{n}^{\ast}\wedge e_{2n}^{\ast}$. The \textit{real symplectic group}, denoted by $\operatorname{Sp}(2n, \mathbb{R})$, is the classical group defined as the set of all those linear endomorphisms of $\mathbb{R}^{2n}$ which preserve $\omega_{\operatorname{cn}}$ and whose Lie algebra, $\operatorname{Lie}(\operatorname{Sp}(2n, \mathbb{R}))$, is therefore the set
$$
{\mathfrak{sp}}(2n, \mathbb{R}) := \{T:\mathbb{R}^{2n} \rightarrow \mathbb{R}^{2n} : \omega_{\operatorname{cn}}(Tu,v) + \omega_{\operatorname{cn}}(u,Tv) = 0 \mbox{ for all } u,v \in \mathbb{R}^{2n}\}.
$$
Let us denote by $C^{2}_{ \omega_{\operatorname{cn}}}(\mathbb{R}^{2n};\mathbb{R}^{2n})$ the vector space $\left\{\mu \in C^{2}(\mathbb{R}^{2n};\mathbb{R}^{2n})  : \operatorname{d}_{\mu}\!\omega_{\operatorname{cn}}= 0 \right\}$, where $\operatorname{d}_{\mu}\!\omega_{\operatorname{cn}}$ is the $3$-form on $\mathbb{R}^{2n}$
defined by the formula
$$
\operatorname{d}_{\mu}\!\omega_{\operatorname{cn}}(v_1,v_2,v_3) =
\displaystyle \sum_{\sigma \in \operatorname{S}_3} \operatorname{sign}(\sigma) \, \omega_{\operatorname{cn}}( \mu(v_{\sigma(1)}, v_{\sigma(2)}), v_{\sigma(3)}).
$$
From this set, we only take all bilinear products on $\mathbb{R}^{2n}$ that endow $\mathbb{R}^{2n}$ with a structure of Lie algebra; we mean the set:
\begin{eqnarray*}
  \mathcal{L}_{\omega_{\tiny{\mbox{cn}}}}(\mathbb{R}^{2n}) &:=& \left\{ \mu \in C^{2}_{\omega_{\tiny{\mbox{cn}}}}(\mathbb{R}^{2n};\mathbb{R}^{2n}) : \operatorname{Jac}(\mu) = 0 \right\}\\
    &=&  \left\{ \mu \in C^{2}(\mathbb{R}^{2n};\mathbb{R}^{2n}) : \operatorname{Jac}(\mu) = 0 \mbox{ and } \operatorname{d}_{\mu}\omega_{\tiny{\mbox{cn}}}=0 \right\}.
\end{eqnarray*}
Here, $\operatorname{Jac}(\mu)$ stands for the alternating trilinear map defined by $$\operatorname{Jac}(\mu)(v_1,v_2,v_3):=\displaystyle \sum_{\sigma \in \operatorname{S}_3} \operatorname{sign}(\sigma) \, \mu( \mu(v_{\sigma(1)}, v_{\sigma(2)}), v_{\sigma(3)});$$
and then unwinding the definition, $\operatorname{Jac}(\mu) = 0$ is the \textit{Jacoby indentity}.

We call $ \mathcal{L}_{\omega_{\tiny{\mbox{cn}}}}(\mathbb{R}^{2n})$ the algebraic set of symplectic Lie algebras, which is a $\operatorname{Sp}(2n, \mathbb{R})$-invariant subset of $C^{2}_{\omega_{\tiny{\mbox{cn}}}}(\mathbb{R}^{2n};\mathbb{R}^{2n})$. The set $\mathcal{L}_{\omega_{\tiny{\mbox{cn}}}}(\mathbb{R}^{2n}) /\operatorname{Sp}(2n, \mathbb{R})$ parametrizes the family of all Lie algebras of dimension $2n$ admitting symplectic structures (up to symplectic isomorphism).

We will consider the Hausdorff vector topology on $C^{2}_{\omega_{\tiny{\mbox{cn}}}}(\mathbb{R}^{2n};\mathbb{R}^{2n})$ and the respective subspace topology on $\mathcal{L}_{\omega_{\tiny{\mbox{cn}}}}(\mathbb{R}^{2n})$ to set up the notion of degeneration. Given a subset $S$ of $C^{2}_{\omega_{\tiny{\mbox{cn}}}}(\mathbb{R}^{2n};\mathbb{R}^{2n})$, we denote by $\overline{S}$ the closure of $S$  with respect to the topology mentioned.

\begin{definition}[Degeneration]
Let $\mu$ and $\lambda$ be two Lie algebra laws in $\mathcal{L}_{\omega_{\tiny{\mbox{cn}}}}(\mathbb{R}^{2n})$. We say that the symplectic Lie algebra \textit{$(\mathbb{R}^{2n}, \mu, {\omega_{\tiny{\mbox{cn}}}} )$ degenerates to $(\mathbb{R}^{2n}, \lambda, {\omega_{\tiny{\mbox{cn}}}} )$ with respect to $\operatorname{Sp}(2n, \mathbb{R})$}, if $\lambda \in \overline{\operatorname{Sp}(2n, \mathbb{R})\cdot \mu}$. Moreover, when $\lambda \notin \operatorname{Sp}(2n, \mathbb{R})\cdot \mu$, then we say that the degeneration is \textit{proper}.
\end{definition}

So, our problem is as follows: given two Lie brackets $\mu$ and $\lambda$ in $\mathcal{L}_{\omega_{\tiny{\mbox{cn}}}}(\mathbb{R}^{2n})$, we need to determine, in an elementary and easily verifiable way, whether $\lambda \in \overline{\operatorname{Sp}(2n, \mathbb{R})\cdot \mu}$ or not. In the standard context of degenerations of (Lie) algebras, in which we must study the closure of $\operatorname{GL}(m, \mathbb{R})$-orbits, it is very usual that familiar lists of invariants of (Lie) algebras are used to face this problem (see, for instance, \cite[Theorem 1]{nesterenko}). In our case, we need to introduce invariants which behave well under degeneration with respect to the action of $\operatorname{Sp}(2n, \mathbb{R})$. One such symplectic invariant follows from the well-known \textit{closed orbit lemma} of Armand Borel (see \cite[Proposition 15.4]{Borel1}) and the respective translation to the real setting due to Michael Jablonski:

\begin{proposition}[{\cite[Proposition 3.2-(b)]{Jablonski1}}]\label{orden}
Let $G$ be a real reductive algebraic group acting linearly and rationally on a real vector space $V$. Then $\partial(G \cdot v) = \overline{G\cdot v} \setminus G\cdot v$ consists of $G$-orbits of strictly smaller dimension.
\end{proposition}

Recall that if $G$ is a Lie group which acts smoothly on a smooth manifold $M$, then each orbit is an immersed submanifold of $M$. Moreover,
If $G$ has at most countably many connected components and $G\cdot p$ is an orbit, then $G\cdot p$ is diffeomorphic to the homogeneous space $G / G_{p}$ (see \cite[Corollary 10.1.16.]{HilgertNeeb}). The group $G_{p}$ considered above is the \textit{isotropy group} or \textit{stabilizer} of $p$, and so, the dimension of $G\cdot p$ is equal to $\operatorname{dim}(\mathfrak{g}) - \operatorname{dim}(\mathfrak{g}_{p})$, where $\mathfrak{g}_{p}$ is the Lie algebra of $G_{p}$.

In our present context, if $G = \operatorname{Sp}(2n, \mathbb{R})$ and $\mu \in \mathcal{L}_{\omega_{\tiny{\mbox{cn}}}}(\mathbb{R}^{2n})$, then $G_{\mu}$ is the group of symplectic automorphism of the symplectic Lie algebra $(\mathbb{R}^{2n} , \mu , {\omega_{\tiny{\mbox{cn}}}})$. Regarding $\operatorname{Lie}(G_{\mu})$, let us introduce the following definition:

\begin{definition}
  Let $(\mathfrak{g}, \omega)$ be a symplectic Lie algebra. A \textit{symplectic derivation} of  $(\mathfrak{g}, \omega)$ is a linear transformation which is a derivation of the Lie algebra $\mathfrak{g}$ and a skew-adjoint linear operator for $\omega$; i.e. a linear map $D: \mathfrak{g} \rightarrow \mathfrak{g}$ such that:
  \begin{itemize}
    \item (Derivation) $D [x,y] = [Dx,y] + [x,Dy]$,
    \item (skew-adjoint for $\omega$) $\omega(Dx ,y) + \omega(x,Dy) = 0$
  \end{itemize}
  for all $x,y \in \mathfrak{g}$.

We denote the real vector space of all symplectic derivations of $(\mathfrak{g}, \omega)$ by $\operatorname{Der}_{\omega}(\mathfrak{g})$.
\end{definition}


Using the last definition and the above-mentioned facts about orbits, an immediate consequence of the Proposition \ref{orden} is that the dimension of the algebra of symplectic derivations of a symplectic Lie algebra is an \textit{obstruction} to study its degenerations:

\begin{corollary}\label{dersymp}
If a symplectic Lie algebra $(\mathbb{R}^{2n}, \mu, {\omega_{\tiny{\mbox{cn}}}} )$ degenerates to $(\mathbb{R}^{2n}, \lambda, {\omega_{\tiny{\mbox{cn}}}} )$ with respect to $\operatorname{Sp}(2n, \mathbb{R})$, then
$$
\operatorname{Dim}(\operatorname{Der}_{\omega_{\tiny{\mbox{cn}}}}(\mathbb{R}^{2n}, \mu ) )< \operatorname{Dim}(\operatorname{Der}_{\omega_{\tiny{\mbox{cn}}}}(\mathbb{R}^{2n}, \lambda ) ).
$$
\end{corollary}

Another important consequence of the Proposition \ref{orden} is that the notion of degeneration of symplectic Lie algebra determines a \textit{partial order} on the orbit space $\mathcal{L}_{\omega_{\tiny{\mbox{cn}}}}(\mathbb{R}^{2n}) / \operatorname{Sp}(2n, \mathbb{R})$ defined by $\operatorname{Sp}(2n, \mathbb{R}) \cdot \mu \leq \operatorname{Sp}(2n, \mathbb{R})\cdot \lambda$ if $(\mathbb{R}^{2n}, \mu, {\omega_{\tiny{\mbox{cn}}}} )$ degenerates to $(\mathbb{R}^{2n}, \lambda, {\omega_{\tiny{\mbox{cn}}}} )$ with respect to $\operatorname{Sp}(2n, \mathbb{R})$.

The following proposition is another useful tool for the study orbit closures.  For all we know, such result is due to Fritz Grunewald and Joyce O'Halloran.

\begin{proposition}[{\cite[Proposition 1.7.]{GrunewaldOHalloran}}]\label{Borel}
Let $\operatorname{G} \subseteq \operatorname{GL}(n,\mathbb{C})$ be a reductive complex algebraic group acting rationally on an algebraic set $\mathcal{Z}$
and let $\operatorname{K} \operatorname{B}$ be an \textit{Iwasawa decomposition} of $\operatorname{G}$. Then for all $p \in \mathcal{Z}$, $\overline{\operatorname{G}\cdot p} = \operatorname{K} \cdot \overline{(\operatorname{B}\cdot p)}$.
\end{proposition}

The same proof of the preceding proposition works as well in the case of continuous representations of real reductive Lie groups.

An Iwasawa decomposition of the real symplectic group $\operatorname{Sp}(2n, \mathbb{R})$ is given by $\operatorname{K} \operatorname{B}$ with
$$
\operatorname{K} =\left\{ \left(\begin{array}{cc} A & -B \\ B & A \end{array} \right) : A + i B \in \operatorname{U}(n)\right\},
$$
$$
\operatorname{A} =\left\{ \operatorname{diag}\left(t_1,\ldots,t_n, \frac{1}{t_1} , \ldots, \frac{1}{t_n}\right) : t_i \in \mathbb{R}_{+} \right\},
$$
$$
\operatorname{N} =\left\{  \left(\begin{array}{cc} (A^{-1})^{\operatorname{T}} & 0 \\ A S & A \end{array} \right) :
\begin{array}{l} A \mbox{ is a unit lower triangular matrix and } \\ S \mbox{ is a symmetric matrix } \end{array}
\right\}
$$
and $\operatorname{B} = \operatorname{A}\operatorname{N}$. Thus, for example, for the case of $\operatorname{Sp}(4, \mathbb{R})$, we have $\operatorname{N}$ is the set
$$
\left\{ \left(\begin{array}{cc} (A^{-1})^{\operatorname{T}} & 0 \\ B & A \end{array} \right) :
 A=\left(\begin{array}{cc} 1 & 0 \\ a & 1 \end{array} \right),\,
 B =  \left(\begin{array}{cc} x & y \\ ax+y & ay+z\end{array} \right), \, a,x,y,z \in \mathbb{R}
 \right\}
$$
As a consequence of having a good description of the group $\operatorname{N}$, we can easily give a parametrization of a $B$-orbit; which could be practical for deciding whether a symplectic Lie algebra degenerates to another one or not. We will illustrate how to implement such parametrizations in the following section.

\subsubsection{Symplectic invariants}\label{invariants}

Since we have a non-degenerate bilinear form $\omega$, we have a natural isomorphisms from the space $L^{k+1}(\mathbb{R}^{2n};\mathbb{R})$ to  $L^{k}(\mathbb{R}^{2n};\mathbb{R}^{2n})$, with $k \in \mathbb{N} \cup \{0\}$, (recall that by convention, $L^{0}(\mathbb{R}^{2n};\mathbb{R}^{2n}) = \mathbb{R}^{2n}$). By analogy with the Riemannian case, for any given $i \in \mathbb{N}$, $1 \leq i \leq k+1$, we can implicitly define a \textit{musical isomorphism} $\sharp_i : L^{k+1}(\mathbb{R}^{2n};\mathbb{R}) \rightarrow L^{k}(\mathbb{R}^{2n};\mathbb{R}^{2n})$ as follows: if $\theta \in L^{k+1}(\mathbb{R}^{2n};\mathbb{R})$, let $\theta^{\sharp_i}$ be the unique multilinear map in $L^{k}(\mathbb{R}^{2n};\mathbb{R}^{2n})$ which satisfies
$$
\theta( v_1,\ldots,v_i,\ldots,v_{k+1} ) = \omega( \theta^{\sharp_i} ( v_1,\ldots,\widehat{v_{i}},\ldots,v_{k+1} ) , v_i ),
$$
where the hat indicates that $v_{i}$ is omitted.

The inverse isomorphism of the \textit{sharp map} is the \textit{flat map} $\flat_i : L^{k}(\mathbb{R}^{2n};\mathbb{R}^{2n}) \rightarrow L^{k+1}(\mathbb{R}^{2n};\mathbb{R})$ given by
$$
\mu^{ \flat_i }( v_1, \ldots, v_i, \ldots, v_{k+1} ) = \omega( \mu(v_1,\ldots,\widehat{v_{i}},\ldots,v_{k+1}) , v_i).
$$

Also, recall that any (Lie) algebra $(\mathbb{R}^{m},\mu)$ allow us to define new multilinear maps and forms which are invariant under isomorphisms. For instance, any \textit{multilinear Lie polynomial} $P_{\mu}$ in $k$ variables $X_1, X_2, \ldots ,X_k$ (with $k\geq 2$) is an element of $L^{k}(\mathbb{R}^{m};\mathbb{R}^{m})$ and it is well-known that $P_{\mu}$  can be written in the form
$$
P_{\mu}(X_1, X_2, \ldots , X_k) = \sum_{ \substack{\sigma \in \operatorname{S}_{k}\\ \sigma(1) = 1 }} a_{\sigma} \: \mu(\ldots \mu(\mu(X_{\sigma(1)},X_{\sigma(2)} ), X_{\sigma(3)}) \ldots X_{\sigma(k)})
$$
(see for instance \cite[\S 5.6.2.]{Reutenauer} or \cite[\S 4.8.1.]{Bahturin}). Furthermore, we can consider the \textit{$i\mbox{-th}$ \textit{trace} of $P_{\mu}$} (see \cite[Appendix B]{Lee2}). This is just the map $\operatorname{tr}_{i}: L^{k+1}(\mathbb{R}^{m};\mathbb{R}^{m}) \rightarrow L^{k}(\mathbb{R}^{m}; \mathbb{R})$ defined by letting  $(\operatorname{tr}_{i} \mu)(v_1, \ldots, v_{k})$ be the trace of the linear map
$$\mu(v_1, \ldots, \underbrace{\square}_{\mbox{\tiny $i$\mbox{-th entry} }}, \ldots , v_{k}) : \mathbb{R}^{m} \rightarrow \mathbb{R}^{m}.$$

Finally, we can also consider tensor of product multilinear forms which are invariant under isomorphisms to obtain new \textit{invariants}.

For a specific example, let $P$ be the non-associative and non-commutative polynomial $P(X,Y,Z) = (X \cdot (Y \cdot Z))$ and consider the map $\mathcal{P}: L^{2}(\mathbb{R}^{m};\mathbb{R}^{m}) \rightarrow  L^{3}(\mathbb{R}^{m};\mathbb{R}^{m})$ given by $\mathcal{P}(\mu)(v_1,v_2,v_3) = \mu(v_1, \mu(v_2,v_3))$. Clearly, $\mathcal{P}$ is a $\operatorname{GL}(m,\mathbb{R})$-equivariant continous map (i.e., $\mathcal{P}(g\cdot \mu) = g\cdot \mathcal{P}(\mu)$) and it allows us to naturally construct the maps
$\operatorname{tr}_1 \circ \mathcal{P}, \operatorname{tr}_3 \circ \mathcal{P}: L^{2}(\mathbb{R}^{m} ;\mathbb{R}^{m}) \rightarrow L^{2}(\mathbb{R}^{m};\mathbb{R})$, which are also $\operatorname{GL}(m,\mathbb{R})$-equivariant continous maps. When $\mathfrak{g}:=(\mathbb{R}^{m},\mu)$ is a Lie algebra, it is to be observed that $(\operatorname{tr}_1 \circ \mathcal{P}) \mu$ is identically zero and $(\operatorname{tr}_3 \circ \mathcal{P})\mu$ is the \textit{Cartan-Killing form} of $\mathfrak{g}$, and $(\operatorname{tr}_3 \circ \mathcal{P})\mu +  c \, (\operatorname{tr}_2\mu) \otimes (\operatorname{tr}_2\mu)$, with $c$ a scalar, is the \textit{modified Cartan-Killing form} (as it was defined in \cite[\S IV.]{nesterenko})

With this motivation in mind,  it is so easy to produce symplectic invariants of symplectic Lie algebra. For instance, it is quite sufficient for our purposes to study the following function: let $c_1,\ldots,c_6 \in \mathbb{R}$ be arbitrary constants and consider $\varphi_{(c_1,\ldots,c_6)}: C^{2}_{\omega_{\tiny{\mbox{cn}}}}(\mathbb{R}^{2n};\mathbb{R}^{2n}) \rightarrow L^{2}(\mathbb{R}^{2n};\mathbb{R}^{2n})$ the function defined implicitly by the relation $\omega ( (\varphi_{(c_1,\ldots,c_6)}\mu)(v_1,v_2) , v_3 )$ equal to
\begin{eqnarray*}
& &    c_1 \, \omega (\mu(v_1,v_2),v_3)  + c_2 \, \omega (\mu(v_2,v_3),v_1) + c_3 \, \omega (\mu(v_3,v_1),v_2) \\
& &   + c_4\,\omega(v_1,v_2) \, (\operatorname{tr}_2\mu)v_3 +  c_5\,\omega(v_2,v_3) \, (\operatorname{tr}_2\mu)v_1 + c_6\,\omega(v_3,v_1) \, (\operatorname{tr}_2\mu)v_2.
\end{eqnarray*}

Note that, there are some redundancies in the definition of $\varphi_{(c_1,\ldots,c_6)}$ because of the definition of the set $C^{2}_{\omega_{\tiny{\mbox{cn}}}}(\mathbb{R}^{2n};\mathbb{R}^{2n})$; for example, $\varphi_{(c_1,c_2,c_3,c_4,c_5,c_6)}$ is equal to $\varphi_{(c_1-c_3,c_2-c_3,0,c_4,c_5,c_6)}$.

The function $\varphi_{(c_1,\ldots,c_6)}$ is continous and $\operatorname{Sp}(2n, \mathbb{R})$-equivariant, and to each algebra law in $C^{2}_{\omega_{\tiny{\mbox{cn}}}}(\mathbb{R}^{2n};\mathbb{R}^{2n})$ it assigns a second algebra structure on $\mathbb{R}^{2n}$.
From a geometrical point of view, if $\mathfrak{g} = (\mathbb{R}^{2n},\mu)$ is a Lie algebra, then $\varphi_{(c_1,\ldots,c_6)}\mu$ defines a \textit{left invariant affine connection} on any connected Lie group with Lie algebra $\mathfrak{g}$ (see \cite[\S 2.4.]{nomizu}). One of the most remarkable example of such an affine connection is given by the \textit{canonical torsion-free flat connection} $\nabla^{{\omega_{\tiny{\mbox{cn}}}}}$ of a symplectic Lie algebra $(\mathbb{R}^{2n},\mu,{\omega_{\tiny{\mbox{cn}}}} )$; which was introduced by Bon-Yao Chu in \cite[Theorem 6.]{chu} and generalized by Oliver Baues and Vicente Cort\'{e}s in \cite[Proposition 1.3.1.]{bauescortes}. The connection $\nabla^{{\omega_{\tiny{\mbox{cn}}}}}$ is given implicitly by the expression
$$
{\omega_{\tiny{\mbox{cn}}}}(\nabla^{{\omega_{\tiny{\mbox{cn}}}}}_{v_1} v_{2} , v_3) = -{\omega_{\tiny{\mbox{cn}}}}(v_2 , \mu(v_1 , v_3));
$$
i.e., $(\varphi_{(0,0,-1,0,0,0)}\mu)(v_1,v_2) = \nabla^{{\omega_{\tiny{\mbox{cn}}}}}_{v_1} v_{2}$.


\subsubsection{symplectic Lie algebras of dimension four}

In \cite{ovando} Gabriela Ovando gives a complete classification of $4$-dimensional symplectic Lie algebras (up to equivalence).
For each symplectic Lie algebra of Ovando's classification, we have fixed a symplectic basis and written the respective structure equations with respect to such basis. We can rewrite such result as:

\begin{proposition}[{\cite[Proposition 2.4.]{ovando}}]
Any four dimensional symplectic Lie algebra is symplectomorphically equivalent to one and only one symplectic Lie algebra $(\mathbb{R}^{4},\mu_{i},{\omega_{\tiny{\mbox{cn}}}})$ where $\mu_{i}$ is given in Table 1 or 2.

\begin{table}[ht]
\centering
\begin{tabular}{ll}
$(\mathfrak{a}_{4}, \omega)$: & $\mu_{0}:=\, $ The four dimensional abelian Lie algebra\\
$(\mathfrak{rh}_{3},\omega)$: & $\mu_{1}:=\left\{ [e_1, e_2] = e_3 \right.$\\
$(\mathfrak{rr}_{3,0},\omega)$: & $\mu_{2}:=\left\{  [e_1, e_3] = e_3 \right.$ \\
$(\mathfrak{rr}_{3,-1},\omega)$: & $\mu_{3}:=\left\{  [e_1, e_2] = -e_2, [e_1, e_4] = e_4 \right.$\\
$(\mathfrak{rr}'_{3,0},\omega)$: & $\mu_{4}:=\left\{  [e_1, e_2] = e_4, [e_1, e_4] = -e_2 \right.$\\
$(\mathfrak{r}_{2}\mathfrak{r}_{2},\omega_{\lambda})$: & $\mu_{5}(\lambda):=\left\{  [e_1, e_2] = -\lambda e_3, [e_1, e_3] = e_3, [e_2, e_4] = e_4, \:
\lambda \geq 0 \right.$\\
$(\mathfrak{r}'_{2},\omega)$: & $\mu_{6}:=\left\{ [e_1, e_3] = e_3, [e_1, e_4] = e_4, [e_2, e_3] = -e_4, [e_2, e_4] = e_3  \right.$\\
$(\mathfrak{n}_{4},\omega)$: & $\mu_{7}:=\left\{  [e_1, e_2] = e_4, [e_1, e_4] = e_3 \right.$\\
$(\mathfrak{r}_{4,0},\omega_{+})$: & $\mu_{8}:=\left\{  [e_1, e_3] = e_3, [e_1, e_4] = e_2 \right.$\\
$(\mathfrak{r}_{4,0},\omega_{-})$: & $\mu_{9}:=\left\{  [e_1, e_3] = e_3, [e_1, e_4] = -e_2 \right.$\\
$(\mathfrak{r}_{4,-1},\omega)$: & $\mu_{10}:=\left\{  [e_1, e_2] = e_2, [e_1, e_3] = -e_3, [e_1, e_4] = -e_3-e_4 \right.$\\


$(\mathfrak{r}_{4,-1,\beta},\omega)$: & $\mu_{11}(\beta):=\left\{  [e_1, e_2] = -e_2, [e_1, e_3] = \beta e_3, [e_1, e_4] = e_4,
\: -1\leq \beta <1  \right.$\\
$(\mathfrak{r}_{4,\alpha,-\alpha},\omega)$: & $\mu_{12}(\alpha):=\left\{ [e_1, e_2] = -e_2, [e_1, e_3] = -\frac{1}{\alpha}e_3, [e_1, e_4] = e_4,
\: -1<\alpha<0 \right.$\\
$(\mathfrak{r}'_{4,0,\delta},\omega_{+})$: & $\mu_{13}(\delta):=\left\{  [e_1, e_2] = -\delta e_4, [e_1, e_3] = e_3, [e_1, e_4] = \delta e_2,
\: \delta>0 \right.$\\
$(\mathfrak{r}'_{4,0,\delta},\omega_{-})$: & $\mu_{14}(\delta):=\left\{  [e_1, e_2] = \delta e_4, [e_1, e_3] = e_3, [e_1, e_4] = -\delta e_2,
\: \delta>0  \right.$\\
\end{tabular}\newline
{\scriptsize  Table 1: Classification of $4$-dimensional symplectic Lie algebras}
\end{table}

\begin{table}[ht]
\centering
\begin{tabular}{ll}
$(\mathfrak{d}_{4,1},\omega_{1})$: & $\mu_{15}:=\left\{ [e_1, e_2] = e_2, [e_1, e_3] = e_3, [e_2, e_4] = e_3  \right.$\\
$(\mathfrak{d}_{4,1},\omega_{2})$: & $\mu_{16}:=\left\{\begin{array}{l} [e_1, e_2] = e_2, [e_1, e_3] = e_3, [e_1, e_4] = e_3, \\ {[e_2, e_4]} = e_3  \end{array} \right.$\\
$(\mathfrak{d}_{4,2},\omega_{1})$: & $\mu_{17}:=\left\{ \begin{array}{l} [e_1, e_2] = 2e_2, [e_1, e_3] = e_3, [e_1, e_4] = -e_4, \\ {[e_2, e_4]} = e_3  \end{array} \right.$\\
$(\mathfrak{d}_{4,2},\omega_{2})$: & $\mu_{18}:=\left\{  \begin{array}{l} [e_1, e_2] = -e_2, [e_1, e_3] = 2e_3, [e_1, e_4] = e_4,\\ {[e_2, e_3]} = e_4 \end{array} \right.$ \\
$(\mathfrak{d}_{4,2},\omega_{3})$: & $\mu_{19}:=\left\{ \begin{array}{l} [e_1, e_2] = -e_2, [e_1, e_3] = 2e_3, [e_1, e_4] = e_4, \\ {[e_2, e_3]} = -e_4  \end{array} \right.$\\


$(\mathfrak{d}_{4,\lambda},\omega)$: & $\mu_{20}(\lambda):=\left\{ \begin{array}{l} [e_1, e_2] = \lambda e_2, [e_1, e_3] = e_3, [e_1, e_4] = (1-\lambda)e_4, \\ {[e_2, e_4]} = e_3, \: \lambda \geq \tfrac{1}{2}, \: \lambda \neq 1,2 \end{array} \right.$\\
$(\mathfrak{d}'_{4,\delta},\omega_{+})$: & $\mu_{21}(\delta) :=\left\{  \begin{array}{l} [e_1, e_2] = \frac{1}{2}\delta e_2-e_4, [e_1, e_3] = \delta e_3, \\ {[e_1, e_4]} = e_2+ \frac{1}{2}\delta e_4,  {[e_2, e_4]} = \delta e_3, \: \delta>0  \end{array} \right.$\\
$(\mathfrak{d}'_{4,\delta},\omega_{-})$: & $\mu_{22}(\delta) :=\left\{  \begin{array}{l}   [e_1, e_2] = -\frac{1}{2}\delta e_2-e_4, [e_1, e_3] = - \delta e_3, \\ {[e_1, e_4]} = e_2- \frac{1}{2}\delta e_4, [e_2, e_4] = -\delta e_3, \: \delta>0 \end{array} \right.$\\
$(\mathfrak{h}_{4},\omega_{+})$: & $\mu_{23}:=\left\{  \begin{array}{l} [e_1, e_2] = \frac{1}{2}e_2, [e_1, e_3] = e_3, [e_1, e_4] = e_2 + \frac{1}{2} e_4, \\ {[e_2, e_4]} = e_3  \end{array} \right.$\\
$(\mathfrak{h}_{4},\omega_{-})$: & $\mu_{24}:=\left\{  \begin{array}{l} [e_1, e_2] = \frac{1}{2}e_2, [e_1, e_3] = e_3, [e_1, e_4] = -e_2+ \frac{1}{2} e_4, \\ {[e_2, e_4]} = e_3 \end{array} \right.$
\end{tabular}\newline
{\scriptsize Table 2: Continuation: Classification of $4$-dimensional symplectic Lie algebras with Heisenberg nilradical}
\end{table}
\end{proposition}

\newpage

\section{Degenerations of $4$-dimensional symplectic Lie algebras}

In this section, we explain how to determine the \textit{Hasse diagram} of degenerations for $4$-dimensional symplectic Lie algebras.
We begin by using Corollary \ref{dersymp} and Proposition \ref{orden} to order the $4$-dimension symplectic Lie algebras by the dimensions of their algebra of derivations and algebra of symplectic derivations (see Table 3). In fact, if $(\mathbb{R}^{4},\mu,{\omega_{\tiny{\mbox{cn}}}})$ degenerates properly to $(\mathbb{R}^{4}, \lambda, {\omega_{\tiny{\mbox{cn}}}} )$ with respect to $\operatorname{Sp}(2n, \mathbb{R})$, then the Lie algebra $(\mathbb{R}^{4},\mu)$ degenerates to $(\mathbb{R}^{4}, \lambda)$ in the usual sense, and so  $\operatorname{Dim}(\operatorname{Der}_{\omega_{\tiny{\mbox{cn}}}}(\mathbb{R}^{2n}, \mu ) )< \operatorname{Dim}(\operatorname{Der}_{\omega_{\tiny{\mbox{cn}}}}(\mathbb{R}^{2n}, \lambda ) )$ and  $\operatorname{Dim}(\operatorname{Der}(\mathbb{R}^{2n}, \mu ) ) \leq \operatorname{Dim}(\operatorname{Der}(\mathbb{R}^{2n}, \lambda ) )$.

\begin{table}[ht]
\centering
\begin{center}
    \begin{tabular}{ c : c: p{3cm} : p{3cm} : p{3cm}}
     $\operatorname{Der}_{\omega_{\tiny{\mbox{cn}}}}  $ & $ \operatorname{Der}$ & \multicolumn{3}{c}{Symplectic Lie Algebra}\\
     \hdashline
     \multirow{1}{*}{$1$} &  5   &  & \multicolumn{1}{c:}{ $(\mathfrak{d}_{4,2},\omega_{2})$, $(\mathfrak{d}_{4,2},\omega_{3})$}  & \\
     \hdashline
     \multirow{5}{*}{$2$} &  4   &   & &\multicolumn{1}{c}{$(\mathfrak{r}_{2}\mathfrak{r}_{2},\omega_{\lambda})$, $(\mathfrak{r}'_{2},\omega)$}   \\
     \cdashline{2-5}
                          &  5   &   \multicolumn{1}{c:}{ $(\mathfrak{d}_{4,\lambda},\omega)$, $(\mathfrak{d}'_{4,\delta},\omega_{\pm})$ }
                                 & \multicolumn{1}{c:}{ $(\mathfrak{d}_{4,2},\omega_{1})$, $(\mathfrak{h}_{4},\omega_{\pm})$,}   &   \multicolumn{1}{c}{ $(\mathfrak{d}_{4,1},\omega_{2})$ }\\
     \cdashline{2-5}
                          &  \multirow{3}{*}{6}
                          & \multicolumn{1}{c:}{ $(\mathfrak{r}_{4,-1,\beta},\omega)$,} &  $(\mathfrak{r}_{4,0},\omega_{\pm})$, $(\mathfrak{r}_{4,-1},\omega)$,
                     & \\
                          &      &  \multicolumn{1}{c:}{  $(\mathfrak{r}_{4,\alpha,-\alpha},\omega)$,} &  $(\mathfrak{rr}_{3,-1},\omega)$, $(\mathfrak{rr}'_{3,0},\omega)$, & \\
                          &       & \multicolumn{1}{c:}{     $(\mathfrak{r}'_{4,0,\delta},\omega_{\pm})$,} &  & \\
                          \hdashline
     \multirow{3}{*}{$3$} &  5  &   & & \multicolumn{1}{c}{ $(\mathfrak{d}_{4,1},\omega_{1})$ } \\
     \cdashline{2-5}
                          &  7  & \multicolumn{3}{c}{$(\mathfrak{n}_{4},\omega)$}   \\
     \cdashline{2-5}
                          &  8  &   & \multicolumn{1}{c:}{$(\mathfrak{r}_{4,-1,-1},\omega)$} &\\
                          \hdashline
     \multirow{2}{*}{$4$} &  7  &   & \multicolumn{1}{c:}{$(\mathfrak{d}_{4,\frac{1}{2}},\omega)$} &\\
     \cdashline{2-5}
                          &  8  & & \multicolumn{2}{c}{ $(\mathfrak{rr}_{3,0},\omega)$ }  \\
                          \hdashline
     \multirow{1}{*}{$5$} &  10  & \multicolumn{3}{c}{$(\mathfrak{rh}_{3},\omega)$}  \\
     \hdashline
     \multirow{1}{*}{$10$} &  16  & \multicolumn{3}{c}{$(\mathfrak{a}_{4},\omega)$}  \\
    \end{tabular}
\end{center}
{\scriptsize Table 3: Dimension of symplectic derivations }
\end{table}

After ordering, we note that to study the degenerations of $(\mathfrak{d}_{4,2},\omega_{2})$ and $(\mathfrak{d}_{4,2},\omega_{3})$
requires considering far more than in the case of other symplectic Lie algebras. Similarly, there are a considerable number of symplectic Lie algebras
that possibly they might degenerate to $(\mathfrak{r}_{4,-1,-1},\omega)$, $(\mathfrak{d}_{4,\frac{1}{2}},\omega)$ or $(\mathfrak{rr}_{3,0},\omega)$.
Only to reduce the number of cases to study, we use the classification of contractions of $4$-dimensional real Lie algebras given in $\cite[\S VIII-B]{nesterenko}$. By such result, we have that the Lie algebra $\mathfrak{d}_{4,2} \cong A_{4,8}^{-\frac{1}{2}}$ only degenerates properly to
$4A_{1} \cong \mathfrak{a}_{4}$, $A_{4,1} \cong \mathfrak{n}_{4}$, $A_{3,1} \oplus A_{1} \cong \mathfrak{rh}_{3}$ and $A_{4,5}^{\frac{1}{2},1,-\frac{1}{2}} \cong \mathfrak{r}_{4,-\frac{1}{2},\frac{1}{2}}$; which allows us to focus on such algebras to determine the degenerations of $(\mathfrak{d}_{4,2},\omega_{2})$ and $(\mathfrak{d}_{4,2},\omega_{3})$.

The same result implies that if a Lie algebra  $\mathfrak{g}$ degenerates properly to:
\begin{itemize}
  \item  $A_{4,5}^{-1,1,1}\cong \mathfrak{r}_{4,-1,-1}$, then  $\mathfrak{g}$ is isomorphic to $A_{4,2}^{-1} \cong \mathfrak{r}_{4,-1}$.
  \item  $A_{2,1}\oplus 2A_{1} \cong \mathfrak{rr}_{3,0}$, then $\mathfrak{g}$ is isomorphic to $A_{4,3} \cong \mathfrak{r}_{4,0}$ or to $2A_{2,1} \cong \mathfrak{r}_{2}\mathfrak{r}_{2}$.
  \item  $A_{4,8}^{1} \cong \mathfrak{d}_{4,\frac{1}{2}}$, then $\mathfrak{g}$ is isomorphic to $A_{4,7} \cong \mathfrak{h}_{4}$.
\end{itemize}

In the remainder of this section we give illustrative examples to show how obtain the classification of orbits closures in the variety of $4$-dimensional symplectic Lie algebras.

\begin{example} Here we describe how to use symplectic invariants defined in the section 2 to prove a non-degeneration result.

\begin{proposition}
The symplectic Lie algebra $(\mathfrak{d}_{4,2},\omega_{2})$ does not degenerate to $(\mathfrak{d}_{4,2},\omega_{1})$ with respect to $\operatorname{Sp}(2n, \mathbb{R})$.
\end{proposition}

\begin{proof}
Suppose on the contrary that $\mu_{17} \in \overline{\operatorname{Sp}(4, \mathbb{R})\cdot \mu_{18}}$. Now, we take $c_1=0$, $c_2=1$, $c_3=0$,
$c_4=-1$, $c_5 = 0$ and $c_6 = -1$, and let $\lambda_{1} = \varphi_{(c_1,\ldots,c_6)} \mu_{17} $ and $\lambda_{2} = \varphi_{(c_1,\ldots,c_6)}\mu_{18}$;
where $\varphi_{(c_1,\ldots,c_6)}$ is the function defined in \S \ref{invariants}.

Then we have $\lambda_{1} \in \overline{\operatorname{Sp}(4, \mathbb{R})\cdot \lambda_{2}}$. In particular $\lambda_{1} \in \overline{\operatorname{GL}(4, \mathbb{R})\cdot \lambda_{2}}$, or in other words, the real algebra $(\mathbb{R}^{4},\lambda_2)$ degenerates to the algebra $(\mathbb{R}^{4},\lambda_{1})$ (in the usual sense of \textit{degenerations of algebras}).

It is straightforward to check that
$$
\lambda_{1} = \left\{
\begin{array}{l}
e_1\ast e_1 = e_1, \, e_1\ast e_2 = -e_2, \, e_1\ast e_3 = e_3, \, e_1\ast e_4 = -e_4, \\
e_2\ast e_1 = 3e_2, \, e_2\ast e_4 = 3e_3
\end{array}
\right.
$$
and
$$
\lambda_2 = \left\{
\begin{array}{l}
 e_2\ast e_1 = e_2, \, e_2\ast e_2 = -e_1, \, e_2\ast e_3 = -e_4, \, e_2\ast e_4 = e_3, \\
 e_4\ast e_1 = 3e_4, \, e_4\ast e_2 = -3e_3.
\end{array}
\right.
$$

Now we consider the $\operatorname{GL}(4, \mathbb{R})$-equivariant function $\mathscr{C} : L^{2}(\mathbb{R}^{4};\mathbb{R}^{4}) \rightarrow L^{2}(\mathbb{R}^{4};\mathbb{R})$ defined by $(\mathscr{C} \: \vartheta)(X,Y) = \operatorname{tr} \vartheta(X, \vartheta(Y , \square))$, and let
$\beta_{1} = \mathscr{C}(\lambda_1)$ and $\beta_{2} = \mathscr{C}(\lambda_2)$. Thus we have the symmetric bilinear form $\beta_{1} \in \overline{\operatorname{GL}(4, \mathbb{R})\cdot \beta_{2}}$ which is a contradiction, since $\beta_{1}$ is a nonzero bilinear form which is positive semidefinite and $\beta_{2}$ is a negative semidefinite bilinear form. This completes the proof.
\end{proof}
\end{example}

\begin{example} The following proposition shows how to take advantage of having a parametrization of a $\operatorname{B}$-orbit to degenerate a symplectic Lie algebra to another.

\begin{proposition}
  The symplectic Lie algebra $(\mathfrak{d}_{4,2},\omega_{2})$ degenerates to $(\mathfrak{r}_{4,-\frac{1}{2},\frac{1}{2}},\omega)$ with respect to $\operatorname{Sp}(2n, \mathbb{R})$.
\end{proposition}

\begin{proof}

By the observation after Proposition \ref{Borel}, we need to find a Lie algebra law $\eta$ such that $\eta \in \overline{\operatorname{B}\cdot \mu_{18}} \cap \operatorname{Sp}(4, \mathbb{R})\cdot \mu_{12}(-\frac{1}{2})$.

It is so easy to write explicitly the elements of the $\operatorname{B}$-orbit of $\mu_{18}$. If $g =\operatorname{diag}\left(t_1,t_2,\frac{1}{t_1} ,\frac{1}{t_2}\right) \in \operatorname{A}$ and $h \in \operatorname{N}$, then $ \xi = (g\cdot h)^{-1} \cdot \mu_{18}$ is of the form:
$$
\begin{array}{l}
\xi({\it e_1},{\it e_2})=-t_{{1}}a{\it e_1}-t_{{1}}{\it e_2}+3\left( ax+y
 \right) t_{{1}}{\it e_3}-
{\frac { 1}{t_{{1}}}} \left( {t_{{1}}}^{2}{a}^{2}x-{t_{{1}}}^{2}ay+x{t_{{2}}}^{2}-2\,{t_{{1}}}^{2}z \right) {\it e_4},\\
{\xi({\it e_1},{\it e_3})}=2\,t_{{1}}{\it e_3}-t_{{1}}a{\it e_4},\:
\xi({\it e_1},{\it e_4})=t_{{1}}{\it e_4},\\
{\xi({\it e_2},{\it e_3})}=-2\,t_{{1}}a{\it e_3}+{\frac { 1}{t_{{1}}}\left( {t_{{2}}}^{2}+{t_{{1}}}^{2}{a}^{2} \right) {\it e_4}},
\: \xi({\it e_2},{\it e_4})=-t_{{1}}a{\it e_4}.
\end{array}
$$
It follows that  $B \cdot \mu$ is a subset of the subspace $W$ of $C^{2}_{\omega_{\tiny{\mbox{cn}}}}(\mathbb{R}^{4};\mathbb{R}^{4})$ given by
$$
W =
\left\{
\begin{array}{l}
\eta({\it e_1},{\it e_2})=b_{{1}}{\it e_1}+b_{{2}}{\it e_2}+b_{{3}}{\it e_3}+b_{{4}}{\it e_4},\\
{\eta({\it e_1},{\it e_3})}=-2\,b_{{2}}{\it e_3}+b_{{1}}{\it e_4}
,\eta({\it e_1},{\it e_4})=-b_{{2}}{\it e_4},\\
{\eta({\it e_2},{\it e_3})}=2\,b_{{1}}{
\it e_3}+b_{{5}}{\it e_4},
\: \eta({\it e_2},{\it e_4})=b_{{1}}{\it e_4}
\end{array} : b_i \in \mathbb{R}
\right\}
$$
It is a simple matter to show that the Lie algebra laws in $W$ that are in the $\operatorname{GL}(4, \mathbb{R})$-orbit of $\mu_{12}(-\frac{1}{2})$ satisfy the condition $b_{{2}}b_{{5}}+{b_{{1}}}^{2} = 0$. With the purpose of getting closer to the $\operatorname{GL}(4, \mathbb{R})$-orbit of $\mu_{12}(-\frac{1}{2})$, we can attempt to restrict our attention to the Lie algebra laws in the $\operatorname{B}$-orbit of $\mu_{18}$ that have the $b_{1}$ component equal to zero, or equivalently, $a=0$:

$$
\begin{array}{l}
\xi({\it e_1},{\it e_2})=-t_{{1}}{\it e_2}+3\,t_{{1}}y{\it e_3}-{\frac { 1}{t_{{1}}}}\left( x{t_{{2}}}^{2}-2\,{t_{{1}}}^{2}z \right) {\it e_4},
{\xi({\it e_1},{\it e_3})}=2\,t_{{1}}{\it e_3},
{\xi({\it e_1},{\it e_4})}=t_{{1}}{\it e_4},\\
{\xi({\it e_2},{\it e_3})}={\frac {{t_{{2}}}^{2}}{t_{{1}}}}{\it e_4}
\end{array}
$$
Note that the $b_{2}$ and $b_{5}$ components of the preceding Lie algebra laws depend only on $t_{1}$ and $t_{2}$, therefore we could take $x=y=z=0$. This motivates to let $g(u) = \operatorname{diag}({1, \exp(u), 1, \exp(-u)})$ with $u \in \mathbb{R}$, and so
$\xi_u = g(u) \cdot \mu_{18}$ is the Lie algebra law:
$$
\xi_u({\it e_1},{\it e_2})=-{\it e_2},
\xi_u({\it e_1},{\it e_3})=2\,{\it e_3},
\xi_u({\it e_1},{\it e_4})={\it e_4},
\xi_u({\it e_2},{\it e_3})={{\rm e}^{-2\,u}}{\it e_4}.
$$
It is clear that $g(u) \in \operatorname{Sp}(4, \mathbb{R})$ and $g(u) \cdot \mu_{18}$  tends to $\mu_{12}(-\frac{1}{2})$ as $u$ tends to $\infty$; which is much stronger than what we wanted to prove.
\end{proof}

\end{example}

\begin{example}  In this final example we presents how to use parametrizations of $\operatorname{B}$-orbits to prove that a symplectic Lie algebra does not degenerate to another.

\begin{proposition}
The symplectic Lie algebras $(\mathfrak{r}_{2}\mathfrak{r}_{2},\omega_{\lambda})$, with $\lambda \geq 0$, and $(\mathfrak{r}'_{2},\omega)$ do not degenerate to $(\mathfrak{n}_{4},\omega)$ with respect to $\operatorname{Sp}(2n, \mathbb{R})$.
\end{proposition}

\begin{proof}
The proof follows from an analysis of the $\operatorname{B}$-orbits of $\mu_{5}(\lambda)$ and $\mu_{6}$. If $g =\operatorname{diag}\left(t_1,t_2,\frac{1}{t_1} ,\frac{1}{t_2}\right) \in \operatorname{A}$ and $h \in \operatorname{N}$, then $\xi = (g\cdot h)^{-1} \cdot \mu_{5}(\lambda)$ is of the form:
$$
\begin{array}{l}
\xi({\it e_1},{\it e_2})= \left( ax+y-t_{{2}}t_{{1}}\lambda \right) t_{{1}}{\it e_3} + \left( -axt_{{2}}-yt_{{2}}-t_{{1}}{a}^{2}x-t_{{1}}ay+{t_{{1
}}}^{2}at_{{2}}\lambda \right) {\it e_4},\\
{\xi({\it e_1},{\it e_3})}=t_{{1}}{\it e_3}-t_{{1}}a{\it e_4},
\: \xi({\it e_2},{\it e_3})=-t_{{1}}a{\it e_3}+a \left( t_{{2}}+t_{{1}}a \right) {\it e_4},
\: \xi({\it e_2},{\it e_4})=t_{{2}}{\it e_4}.
\end{array}
$$

Therefore $\operatorname{B}\cdot \mu_{5}(\lambda)$ and its closure $\overline{\operatorname{B}\cdot \mu_{5}(\lambda)}$ are subsets of the subspace $W$ of $C^{2}_{\omega_{\tiny{\mbox{cn}}}}(\mathbb{R}^{2n};\mathbb{R}^{2n})$ given by
$$
W = \left\{\begin{array}{l}
\eta({\it e_1},{\it  e_2})=b_{{1}}{\it e_3}+b_{{2}}{\it e_4},
\eta({\it e_1},{\it e_3})=b_{{3}}{\it e_3}+b_{{4}}{\it e_4},\\
{\eta({\it e_2},{\it e_3})}=b_{{4}}{\it e_3} +b_{{5}}{\it e_4},
\eta({\it e_2},{\it e_4})=b_{{6}}{\it e_4}
\end{array}
: b_{i} \in \mathbb{R}
\right\}
$$

Now, assume for the sake of contradiction that $ \overline{\operatorname{B}\cdot \mu_{5}(\lambda)} \cap \operatorname{Sp}(4, \mathbb{R})\cdot \mu_{7}$
is nonempty, say, $\eta$ is a Lie algebra law in such intersection. Then there is a sequence $\{\xi_{k}\}_{k \in \mathbb{N}}$ in $\operatorname{B}\cdot \mu_{5}(\lambda)$ converging to $\eta$.

Since $\mathfrak{n}_{4}$ is a \textit{unimodular} Lie algebra, we then focus our attention on unimodular (Lie) algebra laws in the subspace $W$.
An easy computation shows that the $b_3$, $b_4$ and $b_6$ components of any unimodular algebra law in $W$ must be zero.
And so $\eta$ is of the form
$$
\eta({\it e_1},{\it e_2})=c_{{1}}{\it e_3}+c_{{2}}{\it e_4}, \: \eta({\it e_2},{\it e_3})=c_{{5}}{\it e_4}
$$
and, better yet, since $(\mathbb{R}^{4}, \eta)$ is isomorphic to $\mathfrak{n}_{4}$, it follows that $\eta(e_1,e_2)$ and $\eta(e_2,e_3)$ are linearly independent.

Let us write $b_{i}(k)$ for the $b_{i}$-component of $\xi_{k}$ for $i=1, \ldots, 6$. As $\xi_{k} \rightarrow \eta$ as $n \rightarrow +\infty$ we have
$b_{3}(k)$, $b_{4}(k)$ and $b_{6}(k)$ tend to $0$ as $k$ tends to $+\infty$. On the other hand, we have $\operatorname{det}( \xi_{k}({\it e_1},{\it e_2}) , \xi_{k}({\it e_2},{\it e_3}) ) = b_{1}(k) b_{5}(k) - b_{2}(k)b_{4}(k)$ is equal to $\lambda \ b_{3}(k) b_{4}(k) b_{6}^{2}(k)$; which can be deduced from the form of elements in $\operatorname{B}\cdot \mu_{5}(\lambda)$. Therefore
\begin{eqnarray*}
  \operatorname{det}(\eta({\it e_1},{\it e_2}) , \eta({\it e_2},{\it e_3}) )
  & = & \lim_{k \to +\infty}\operatorname{det}( \xi_{k}({\it e_1},{\it e_2}) , \xi_{k}({\it e_2},{\it e_3}) )\\
  & = & \lim_{k \to +\infty} \lambda \ b_{3}(k) b_{4}(k) b_{6}^{2}(k)\\
  & = & 0,
\end{eqnarray*}
and consequently $\eta({\it e_1},{\it e_2})$ and $\eta({\it e_2},{\it e_3})$ are linearly dependent, which is a contradiction.

The preceding proof can be summarized by saying that $\overline{\operatorname{B}\cdot \mu_{5}(\lambda)}$ is contained in the algebraic set
$$
\mathcal{Z}(\lambda) = \{ \xi \in W : b_{{1}} b_{{5}} - b_{{2}}b_{{4}} - \lambda\,b_{{3}}b_{{4}}{b_{{6}}}^{2} = 0\},
$$
while $\operatorname{GL}(4, \mathbb{R}) \cdot \mu_{7}$ does not intersect $\mathcal{Z}(\lambda)$.

The proof of the case with $(\mathfrak{r}'_{2},\omega)$  is notably much simpler than the preceding proof. It is straightforward to check that
$\operatorname{B}\cdot \mu_{6}$ is contained in the subspace $\widetilde{W} $ of $C^{2}_{\omega_{\tiny{\mbox{cn}}}}(\mathbb{R}^{2n};\mathbb{R}^{2n})$ defined by
$$
\widetilde{W} =
\left\{
\begin{array}{l}
\xi({\it e_1},{\it e_2})=b_{{1}}{\it e_4},
\xi({\it e_1},{\it e_3})=b_{{2}}{\it e_3},
\xi({\it e_1},{\it e_4})=b_{{2}}{\it e_4},\\
\xi({\it e_2},{\it e_3})=b_{{3}}{\it e_4},
\xi({\it e_2},{\it e_4})=b_{{2}}{\it e_3}+b_{{4}}{\it e_4}
\end{array} : b_i \in \mathbb{R}
\right\}
$$
and if $\eta$ is a unimodular (Lie) algebra law in $\widetilde{W}$, then its $b_2$ and $b_4$ component are zero, and so the dimension of the \textit{derived algebra} of $(\mathbb{R}^{4}, \eta)$ is less than or equal to $1$. Since the dimension of derived algebra of $\mathfrak{n}_{4}$ is $2$,
we have  $\operatorname{GL}(4, \mathbb{R}) \cdot \mu_{7}$ not intersect $\widetilde{W} \supseteq \overline{\operatorname{B}\cdot \mu_{6}}$. This completes the proof.
\end{proof}

\end{example}

By using arguments similar to those used above we obtain the following theorem:\newline

\noindent \textbf{{Theorem A}}  \textit{The Hasse diagram of degenerations in $\mathcal{L}_{\omega_{\tiny{\mbox{cn}}}}(\mathbb{R}^{4})$ is given by the union the following diagrams:}

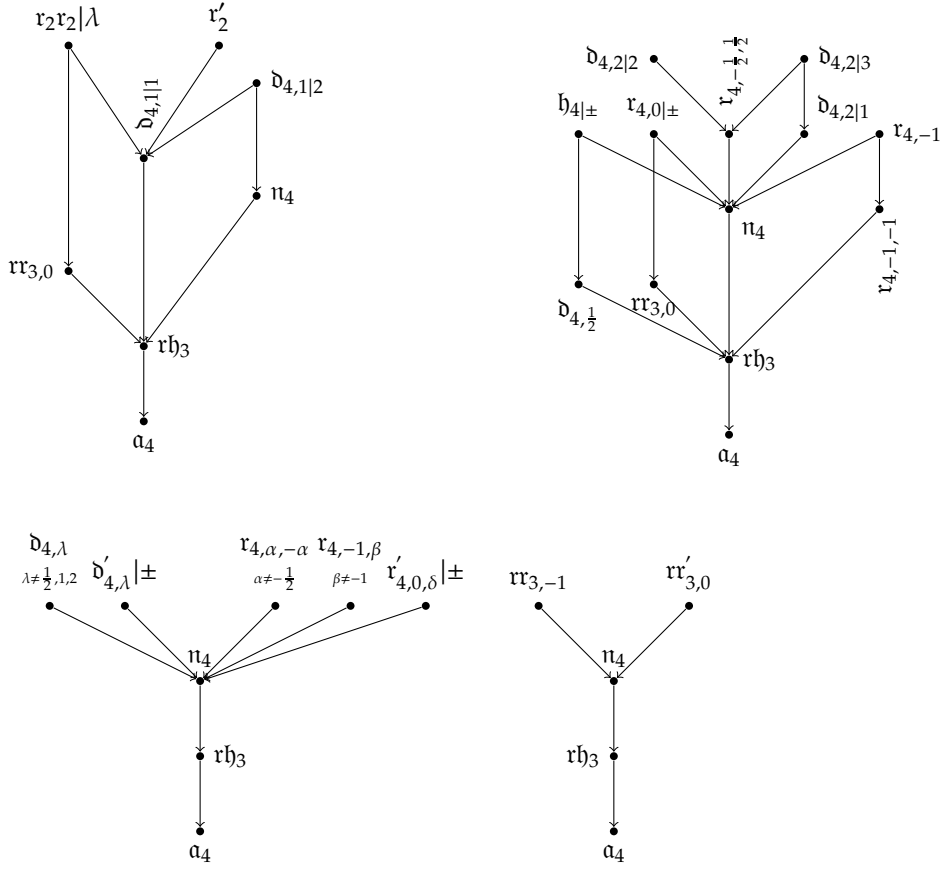
\begin{enumerate}
\begin{multicols}{2}
\item[]
\begin{tikzpicture}
  \node[circle,fill,inner sep=0pt,minimum size=3pt,label=above:{$\mathfrak{r}_{2}\mathfrak{r}_2 | \lambda$}] (r2r2) at (-1,0) {};
  \node[circle,fill,inner sep=0pt,minimum size=3pt,label=above:{$\mathfrak{r}'_{2}$}] (r2prima) at (1,0) {};
  \node[circle,fill,inner sep=0pt,minimum size=3pt,label=right:{$\mathfrak{d}_{4,1|2}$}] (d412) at (1.5,-0.5) {};
  \node[circle,fill,inner sep=0pt,minimum size=3pt,label={[label distance=0.25cm,rotate=90]right:{$\mathfrak{d}_{4,1|1}$}}] (d411) at (0,-1.5) {};
  \node[circle,fill,inner sep=0pt,minimum size=3pt,label=right:{$\mathfrak{n}_{4}$}] (n4) at (1.5,-2) {};
  \node[circle,fill,inner sep=0pt,minimum size=3pt,label=left:{$\mathfrak{r}\mathfrak{r}_{3,0}$}] (rr30) at (-1,-3) {};
  \node[circle,fill,inner sep=0pt,minimum size=3pt,label=right:{$\mathfrak{r}\mathfrak{h}_{3}$}] (rh3) at (0,-4) {};
    \node[circle,fill,inner sep=0pt,minimum size=3pt,label=below:{$\mathfrak{a}_{4}$}] (a4) at (0,-5) {};
\draw [->] (r2r2) edge (d411) (r2prima) edge (d411);
\draw [->] (d412) edge (d411) (d412) edge (n4) (d411)edge (rh3);
\draw [->] (n4) edge (rh3);
\draw [->] (r2r2) edge (rr30);
\draw [->] (rr30) edge (rh3);
\draw [->] (rh3) edge (a4);
\end{tikzpicture}

\item[]
\begin{tikzpicture}
\node[label=above:$\,$] (emptyyy) at (0,0.5){};
 \node[circle,fill,inner sep=0pt,minimum size=3pt,label=left:{$\mathfrak{d}_{4,2|2}$}] (d422) at (-1,0) {};
 \node[circle,fill,inner sep=0pt,minimum size=3pt,label=right:{$\mathfrak{d}_{4,2|3}$}] (d423) at (+1,0) {};
 \node[circle,fill,inner sep=0pt,minimum size=3pt,label={[label distance=0.25cm,rotate=90]right:{$\mathfrak{r}_{4,-\frac{1}{2},\frac{1}{2}}$}}] (r4menos1unmedio) at (0 ,-1 ) {};
 \node[circle,fill,inner sep=0pt,minimum size=3pt,label=above:{$\mathfrak{h}_{4|\pm}$}] (h4masmenos) at ( -2, -1) {};
 \node[circle,fill,inner sep=0pt,minimum size=3pt,label=above:{$\mathfrak{r}_{4,0|\pm}$}] (r4masmenos) at (-1 ,-1 ) {};
 \node[circle,fill,inner sep=0pt,minimum size=3pt,label=above right:{$\mathfrak{d}_{4,2|1}$}] (d421) at ( 1, -1) {};
 \node[circle,fill,inner sep=0pt,minimum size=3pt,label=right:{$\mathfrak{r}_{4,-1}$}] (r4menos1) at ( 2, -1) {};
 \node[circle,fill,inner sep=0pt,minimum size=3pt,label=below right:{$\mathfrak{n}_{4}$}] (n4) at (0 ,-2 ) {};
 \node[circle,fill,inner sep=0pt,minimum size=3pt] (r4menos1menos1) at (2 ,-2 ) {};
 \node[ label={[label distance=-1.25cm,rotate=90]right:{$\mathfrak{r}_{4,-1,-1}$}}] (emptyyy) at (2,-2){};

 \node[circle,fill,inner sep=0pt,minimum size=3pt,label=below:{$\mathfrak{d}_{4,\frac{1}{2}}$}] (d4unmedio) at ( -2,-3 ) {};
 \node[circle,fill,inner sep=0pt,minimum size=3pt,label=below:{$\mathfrak{rr}_{3,0}$}] (rr3cero) at ( -1,-3 ) {};
 \node[circle,fill,inner sep=0pt,minimum size=3pt,label=right:{$\mathfrak{r}\mathfrak{h}_{3}$}] (rh3) at ( 0,-4 ) {};
 \node[circle,fill,inner sep=0pt,minimum size=3pt,label=below:{$\mathfrak{a}_{4}$}] (a4) at ( 0,-5 ) {};

\draw [->] (d422) edge (r4menos1unmedio) (d423) edge (r4menos1unmedio);
\draw [->] (d423) edge (d421);
\draw [->] (d421) edge (n4);
\draw [->] (r4menos1unmedio) edge (n4) (r4menos1) edge (n4);
\draw [->] (r4menos1) edge (r4menos1menos1);
\draw [->] (r4menos1menos1) edge (rh3);
\draw [->] (r4masmenos) edge (n4) (r4masmenos) edge (rr3cero);
\draw [->] (h4masmenos) edge (d4unmedio) (h4masmenos) edge (n4);
\draw [->] (n4) edge (rh3) (d4unmedio) edge (rh3) (rr3cero)  edge (rh3);
\draw [->] (rh3) edge (a4);
\end{tikzpicture}

\end{multicols}
\end{enumerate}

\begin{enumerate}
\begin{multicols}{2}

\item[]
\begin{tikzpicture}
\node[circle,fill,inner sep=0pt,minimum size=3pt,label=above:{$\begin{array}{c}  \mathfrak{d}_{4,\lambda} \\ {\scriptscriptstyle \lambda\neq \frac{1}{2},1,2} \end{array}$}] (d4lambda) at (-2,0) {};
\node[circle,fill,inner sep=0pt,minimum size=3pt,label=above:{$\mathfrak{d}^{'}_{4,\lambda}|\pm$}] (d4primalambda) at (-1,0) {};
\node[circle,fill,inner sep=0pt,minimum size=3pt,label=above:{$\begin{array}{c} \mathfrak{r}_{4,\alpha,-\alpha}  \\ {\scriptscriptstyle \alpha \neq -\frac{1}{2}} \end{array}$}] (r4alpha) at (+1,0) {};
\node[circle,fill,inner sep=0pt,minimum size=3pt,label=above:{$\begin{array}{c}  \mathfrak{r}_{4,-1,\beta} \\ {\scriptscriptstyle \beta \neq -1} \end{array}$}] (r4beta) at (+2,0) {};
\node[circle,fill,inner sep=0pt,minimum size=3pt,label=above:{$ \mathfrak{r}^{'}_{4,0,\delta} | \pm $}] (r4betaprima) at (+3,0) {};
\node[circle,fill,inner sep=0pt,minimum size=3pt,label=above:{$\mathfrak{n}_{4}$}] (n4) at (0,-1) {};
\node[circle,fill,inner sep=0pt,minimum size=3pt,label=right:{$\mathfrak{r}\mathfrak{h}_{3}$}] (rh3) at (0,-2) {};
\node[circle,fill,inner sep=0pt,minimum size=3pt,label=below:{$\mathfrak{a}_{4}$}] (a4) at (0,-3) {};


\node[circle,fill,inner sep=0pt,minimum size=3pt,label=above:{$\mathfrak{r}\mathfrak{r}_{3,-1}$}] (r3menos) at (+4.5,+0) {};
\node[circle,fill,inner sep=0pt,minimum size=3pt,label=above:{$\mathfrak{r}\mathfrak{r}^{'}_{3,0}$}] (r3prima) at (+6.5,+0) {};
\node[circle,fill,inner sep=0pt,minimum size=3pt,label=above:{$\mathfrak{n}_{4}$}] (n4B) at (+5.5,-1) {};
\node[circle,fill,inner sep=0pt,minimum size=3pt,label=left:{$\mathfrak{r}\mathfrak{h}_{3}$}] (rh3B) at (+5.5,-2) {};
\node[circle,fill,inner sep=0pt,minimum size=3pt,label=below:{$\mathfrak{a}_{4}$}] (a4B) at (+5.5,-3) {};


\draw [->]  (d4lambda) edge(n4)  (d4primalambda) edge(n4)  (r4alpha) edge(n4)  (r4beta) edge(n4) (r4betaprima) edge(n4);
\draw [->]  (n4) edge (rh3);
\draw [->] (rh3) edge (a4);


\draw [->] (r3menos) edge (n4B) (r3prima) edge (n4B);
\draw [->] (n4B) edge(rh3B);
\draw [->] (rh3B) edge(a4B);

\end{tikzpicture}

\end{multicols}
\end{enumerate}

\section{Applications}

Some geometric quantities vary ``continuously'' as functions of tensor spaces. This vague assertion gains in interest if we consider invariant geometric structures on \textit{homogeneous spaces} in which such structures are determined by a tensor defined on a fixed tangent space.

If we have a Lie bracket $\mu \in \mathcal{L}_{\omega_{\tiny{\mbox{cn}}}}(\mathbb{R}^{2n})$ and, in consequence, a symplectic Lie algebra $(\mathbb{R}^{2n},\mu,{\omega_{\tiny{\mbox{cn}}}})$, then we automatically have an almost-K\"ahler structure on the Lie algebra $(\mathbb{R}^{2n},\mu)$ given by the triple $(\langle \cdot, \cdot \rangle_{\tiny{\mbox{cn}}} , {\omega_{\tiny{\mbox{cn}}}}, J_{\tiny{\mbox{cn}}})$, where $\langle \cdot, \cdot \rangle_{\tiny{\mbox{cn}}}$ denotes the (standard) dot product on $\mathbb{R}^{2n}$ and $J_{\tiny{\mbox{cn}}}$ is the linear map defined by $J_{\tiny{\mbox{cn}}} e_{k} = e_{n+k}$ and $J_{\tiny{\mbox{cn}}}e_{n+k} = -e_k$ for $k = 1,\ldots, n$.

On the other hand, suppose $V$ is a real vector space, $\langle \cdot,\cdot \rangle$ is a positive-definite inner product on $V$ and $\omega$ is a symplectic structure on $V$. If $J:V\rightarrow V$ defined implicitly by $\omega(\cdot, \centerdot ) =  \langle J\cdot, \centerdot \rangle$ is such that $J^2 = -\operatorname{Id}_{V}$, then it is easy to prove that there exists an ordered orthonormal basis for $(V,\langle \cdot,\cdot \rangle)$ of the form $(e_1, \ldots, e_n , Je_1, \ldots Je_n)$, which is clearly a symplectic basis for $(V,\omega)$ (since $J$ is a \textit{complex structure} on $V$, recall that it allows us to regard $V$ as a complex vector space and then we can consider the positive-definite Hermitian inner product on $(V,J)$ defined by $\langle \! \langle \cdot , \centerdot \rangle \! \rangle = \langle \cdot , \centerdot \rangle + \sqrt{-1}\omega(\cdot , \centerdot)$; which admits an orthonormal basis). Thus, if $(\langle \cdot,\cdot \rangle , \omega, J)$ is an almost K\"ahler structure on a Lie algebra $\mathfrak{g}$, then $(\mathfrak{g},\langle \cdot,\cdot \rangle , \omega, J)$ is equivalent to $(\mathbb{R}^{2n}, \mu, \langle \cdot, \cdot \rangle_{\tiny{\mbox{cn}}} , {\omega_{\tiny{\mbox{cn}}}}, J_{\tiny{\mbox{cn}}} )$ with $\mu \in \mathcal{L}_{\omega_{\tiny{\mbox{cn}}}}(\mathbb{R}^{2n})$.

Therefore, when we ``walk along'' the $\operatorname{Sp}(2n, \mathbb{R})$-orbit of a Lie algebra law $\mu \in \mathcal{L}_{\omega_{\tiny{\mbox{cn}}}}(\mathbb{R}^{2n})$, we obtain different compatible inner products with the symplectic Lie algebra $(\mathbb{R}^{2n},\mu,{\omega_{\tiny{\mbox{cn}}}})$ and the Hasse diagram of $\mathcal{L}_{\omega_{\tiny{\mbox{cn}}}}(\mathbb{R}^{2n})$
is a kind of map that shows how the $\operatorname{Sp}(2n, \mathbb{R})$-orbits are interconnected.

We can use the ideas presented in the preceding sections to study evolution of left invariant almost-K\"ahler structures on Lie groups under geometric flows, cohomological aspects or deformation theory of left invariant symplectic structures on Lie Groups, among other significant potential applications.
In this section, we give an application of \textbf{Theorem A} to the problem of studying curvature properties of left invariant almost K\"ahler structures on Lie groups.

We begin by defining the continuous function $\:\: \square^{\tiny{\mbox{$\mathscr{LC}$}}}\!: C^{2}(\mathbb{R}^{n};\mathbb{R}^{n}) \rightarrow L^{2}(\mathbb{R}^{n};\mathbb{R}^{n})$ defined implicitly by
$$
2 \langle (\mu^{{\tiny{\mbox{$\mathscr{LC}$}}}}(v_1,v_2) , v_3 \rangle_{\tiny{\mbox{cn}}} =
\langle \mu(v_1 , v_2) ,v_3 \rangle_{\tiny{\mbox{cn}}}
-\langle \mu(v_2 , v_3) ,v_1 \rangle_{\tiny{\mbox{cn}}}
+\langle \mu(v_3 , v_1) ,v_2 \rangle_{\tiny{\mbox{cn}}}.
$$
If $(\mathbb{R}^{n},\mu)$ is a Lie algebra, then $\mu^{{\tiny{\mbox{$\mathscr{LC}$}}}}$ is the Levi-Civita connection of  $(\mathbb{R}^{n},\mu, \langle \cdot,\cdot\rangle_{\tiny{\mbox{cn}}})$.

Now consider the continuous functions $\mathscr{R}: C^{2}(\mathbb{R}^{n};\mathbb{R}^{n}) \rightarrow L^{3}(\mathbb{R}^{n};\mathbb{R}^{n})$ given by
\begin{eqnarray*}
(\mathscr{R}\mu)(v_1,v_2,v_3) &=&
\mu^{{\tiny{\mbox{$\mathscr{LC}$}}}}(v_1, \mu^{{\tiny{\mbox{$\mathscr{LC}$}}}}(v_2,v_3))
-\mu^{{\tiny{\mbox{$\mathscr{LC}$}}}}(v_2, \mu^{{\tiny{\mbox{$\mathscr{LC}$}}}}(v_1,v_3))\\
& &-\mu^{{\tiny{\mbox{$\mathscr{LC}$}}}}(\mu(v_1,v_2),v_3))
\end{eqnarray*}
and its $2\mbox{-nd}$ trace, $\operatorname{Ric} =  \operatorname{tr}_2 \circ \mathscr{R}: C^{2}(\mathbb{R}^{n};\mathbb{R}^{n}) \rightarrow L^{2}(\mathbb{R}^{n};\mathbb{R}) $; i.e. $(\operatorname{Ric} \mu)(v_1,v_3) = \operatorname{trace}((\mathscr{R}\mu)(v_1,\cdot,v_3))$. When $(\mathbb{R}^{n},\mu)$ is a Lie algebra, $\mathscr{R}\mu$ and $\operatorname{Ric}\mu$ are the \textit{Riemann curvature tensor} and the \textit{Ricci curvature tensor} of $(\mathbb{R}^{n},\mu, \langle \cdot,\cdot\rangle_{\tiny{\mbox{cn}}})$, respectively.

It follows from \textbf{Theorem A} that almost any four dimensional symplectic Lie algebra degenerates to $(\mathfrak{n}_{4},\omega)$. In the following lemma we study the Ricci curvature of some compatible inner products with $(\mathfrak{n}_{4},\omega)$ and other symplectic Lie algebras, which will
be useful in our proof of \textbf{Theorem B}.

\begin{lemma}\label{lemma}
The symplectic Lie algebras $(\mathfrak{n}_{4},\omega)$, $(\mathfrak{d}_{4,\frac{1}{2}},\omega)$, $(\mathfrak{d}_{4,1},\omega_{1})$ and  $(\mathfrak{r}_{4,-1,-1},\omega)$  admits compatible inner products with nondegenerate Ricci form of signature $(-,-,-,+)$.
\end{lemma}

\begin{proof}
Put $g(t) = \operatorname{diag}(\frac{1}{t},1,t,1) \in \operatorname{Sp}(4, \mathbb{R})$ and let $\xi_{t}:= g(t)\cdot \mu_{7}$. Then $(\mathfrak{n}_{4},\omega)$ is symplectomorphically equivalent to $(\mathbb{R}^{4},\xi_{t}, {\omega_{\tiny{\mbox{cn}}}})$ with:
$$
\xi_{t}= \left\{
\xi_{t} ({e_1},{e_2}) =t{e_4}, \:
\xi_{t} ({e_1},{e_4}) ={t}^{2}{e_3}
\right.
$$

Now we endow $(\mathbb{R}^{4},\xi_{t}, {\omega_{\tiny{\mbox{cn}}}})$ with the compatible inner product $\langle \cdot, \cdot \rangle_{\tiny{\mbox{cn}}}$
and calculate the respective Ricci form. For computational purposes, it is much more convenient to use the well-known formula for computing the Ricci form of a metric Lie algebra $(\mathbb{R}^{m},\eta,\langle \cdot, \cdot \rangle_{\tiny{\mbox{cn}}})$; which is due to Dmitrii V. Alekseevski\u{\i} as far as we know  (see \cite[formula (2.5)]{Alekseevskii}). If $(\mathbb{R}^{m},\eta)$ is a nilpotent Lie algebra, then the formula reduces to
$$
(\operatorname{Ric}\eta)(v,v) = \displaystyle -\frac{1}{2} \sum_{i,j} \langle \eta(v,e_i),e_j \rangle_{\tiny{\mbox{cn}}}^{2} + \frac{1}{2} \sum_{i<j} \langle \eta(e_i,e_j),v \rangle_{\tiny{\mbox{cn}}}^{2}.
$$

Therefore, we have that the matrix of the Ricci form of $(\mathbb{R}^{4},\xi_{t},\langle \cdot,\cdot \rangle_{\tiny{\mbox{cn}}}, {\omega_{\tiny{\mbox{cn}}}})$ with respect to the canonical basis of $\mathbb{R}^4$ is the diagonal matrix:
$$
\operatorname{diag}\left(-\frac{1}{2} {t}^{2} - \frac{1}{2} {t}^{4} , -\frac{1}{2} {t}^{2}, \frac{1}{2}{t}^{4} ,
\frac{1}{2}{t}^{2} - \frac{1}{2} {t}^{4}\right),
$$
and so the signature of the Ricci form of $(\mathbb{R}^{4},\xi_{t},\langle \cdot,\cdot \rangle_{\tiny{\mbox{cn}}}, {\omega_{\tiny{\mbox{cn}}}})$ is $(-,-,+,+)$ if $0<|t|<1$, and it is $(-,-,-,+)$ if $|t|>1$.

To study the cases in which the symplectic Lie algebra is $(\mathfrak{d}_{4,\frac{1}{2}},\omega)$ or $(\mathfrak{d}_{4,1},\omega_{1})$, we consider the linear transformations $h(t) \in \operatorname{Sp}(4, \mathbb{R})$ defined by
$$h(t)e_1 = e_1-t e_4, \; h(t)e_2 = e_2 -te_3, \; h(t)e_3 = e_3 \mbox{ and } h(t)e_4= e_4,$$
and let $\rho_t = h(t)\cdot \mu_{20}(\frac{1}{2})$ and $\varrho_{t} = h(t)\cdot \mu_{15}$. We have $(\mathfrak{d}_{4,\frac{1}{2}},\omega)$ and $(\mathfrak{d}_{4,1},\omega_{1})$ are symplectomorphically equivalent to $(\mathbb{R}^{4},\rho_{t}, {\omega_{\tiny{\mbox{cn}}}})$ and $(\mathbb{R}^{4},\varrho_{t}, {\omega_{\tiny{\mbox{cn}}}})$,  respectively, where:
$$
\rho_{t} = \left\{
\rho_{t}({\it e_1},{\it e_2})=\frac{1}{2}\,{\it e_2} - \frac{1}{2}\,t{\it e_3},\,
\rho_{t}({\it e_1},{\it e_3})={\it e_3},\,
\rho_{t}({\it e_1},{\it e_4})=\frac{1}{2}\,{\it e_4},\,
\rho_{t}({\it e_2},{\it e_4})={\it e_3}
\right.
$$
and
$$
\varrho_{t}= \displaystyle \left\{
\varrho_{t}({\it e_1},{\it e_2})={\it e_2}-t{\it e_3},\,
\varrho_{t}({\it e_1},{\it e_3})={\it e_3},\,
\varrho_{t}({\it e_2},{\it e_4})={\it e_3} \phantom{\frac{1}{1}}
\right.
$$
and by \cite[formula (2.5)]{Alekseevskii}) again, it is easy to verify that the signature of the Ricci form of $(\mathbb{R}^{4},\rho_{12}, \langle \cdot, \cdot \rangle_{\tiny{\mbox{cn}}}, {\omega_{\tiny{\mbox{cn}}}})$ and $(\mathbb{R}^{4},\varrho_{2}, \langle \cdot, \cdot \rangle_{\tiny{\mbox{cn}}}, {\omega_{\tiny{\mbox{cn}}}})$ is $(-,-,-,+)$.  It is worth noting that $(\langle \cdot, \cdot \rangle_{\tiny{\mbox{cn}}}, {\omega_{\tiny{\mbox{cn}}}})$ determines on $\mathfrak{d}_{4,\frac{1}{2}} \cong (\mathbb{R}^{4},\rho_{0})$ an \textit{Einstein almost K\"ahler structure} with negative scalar curvature.

Finally, the $(\mathfrak{r}_{4,-1,-1},\omega)$ case is the simplest one to handle because the canonical inner product $\langle \cdot, \cdot \rangle_{\tiny{\mbox{cn}}}$ is compatible with $(\mathbb{R}^{4}, \mu_{11}(-1) , {\omega_{\tiny{\mbox{cn}}}})$ and the matrix of the Ricci form of $(\mathbb{R}^{4}, \mu_{11}(-1) , \langle \cdot, \cdot \rangle_{\tiny{\mbox{cn}}})$ with respect to the canonical basis is $\operatorname{diag}(-3,-1,-1,1)$.
\end{proof}

\begin{remark}
The Proposition 22 in \cite{KremlevNikonorov} says that the signature of the Ricci form of all inner products on the Lie algebra $A^{1}_{4,9}$ is $(-,-,-,+)$.
This statement is incorrect, since the Lie algebra $\mathfrak{d}_{4,\frac{1}{2}} \cong (\mathbb{R}^{4},\rho_{0})$ is isomorhphic to $A^{1}_{4,9}$ and the signature of the Ricci form of $(\mathbb{R}^{4},\rho_{0}, \langle \cdot, \cdot \rangle_{\tiny{\mbox{cn}}})$ is $(-,-,-,-)$ and
there is a $\widehat{t} \in (0,12)$ such that the Ricci form of $(\mathbb{R}^{4},\rho_{\widehat{t}}, \langle \cdot, \cdot \rangle_{\tiny{\mbox{cn}}})$ has signature $(-,-,-,0)$.

.\newline
\end{remark}

\begin{proofB}
Let $(\mathbb{R}^{4},\mu_{i},{\omega_{\tiny{\mbox{cn}}}})$ be a $4$-dimensional symplectic Lie algebra with $i \neq 0,1,2$ and let
$A = \{ \xi_{2} , \rho_{12} , \varrho_{2}, \mu_{11}(-1) \}$ the set of the Lie algebra laws obtained in the preceding lemma.

From \textbf{Theorem A},  it is clear that $(\mathbb{R}^{4},\mu_{i},{\omega_{\tiny{\mbox{cn}}}})$ degenerates to $(\mathbb{R}^{4}, \zeta, {\omega_{\tiny{\mbox{cn}}}})$ with respect to $\operatorname{Sp}(4, \mathbb{R})$, for some $\zeta \in A$. So, there exists a sequence the symplectic transformations $\{s_{k}\}_{k\in \mathbb{N}} \subseteq \operatorname{Sp}(4, \mathbb{R})$ such that $s_{k} \cdot \mu_{i}$  tends to $\zeta$ as $k$ tends to $+\infty$. Since the function $\operatorname{Ric}$ is continuous, we have $\operatorname{Ric}s_{k} \cdot \mu_{i}$ tends to $\operatorname{Ric} \zeta$ as $k$ tends to $+\infty$.

The previous lemma shows that $\operatorname{Ric}\zeta$ has nondegenerate signature $(-,-,-,+)$ and since such signature is an \textit{open condition} in the vector space of all symmetric bilinear forms, $S^{k}(V^{\ast})$, (see for instance the very nice survey \cite{GhysRanicki}), we have an open neighborhood $U \subseteq S^{k}(V^{\ast})$ containing $\operatorname{Ric}\zeta$ such that for each bilinear form $\kappa \in U$, the signature of $\kappa$ is $(-,-,-,+)$.
It follows from the convergence of the sequence $\{\operatorname{Ric}s_{k} \cdot \mu_{i} \}$ that there exists an index $N$ so that $\operatorname{Ric}s_{N} \cdot \mu_{i}$ has signature $(-,-,-,+)$ or in other words, $\langle \cdot, \cdot \rangle_{\tiny{\mbox{cn}}}$ is a compatible inner product with $(\mathbb{R}^{4}, s_{N} \cdot \mu_{i}, {\omega_{\tiny{\mbox{cn}}}} )$ with such signature.

Conversely, it is straightforward to prove that the Ricci form of any inner product on $\mathfrak{h}_{3} \times \mathbb{R}$, $\mathfrak{aff}(\mathbb{R})\times \mathbb{R}^2$ or $\mathfrak{a}_{4}$ is degenerate.
\end{proofB}

\section{Remarks}
As we mentioned above, the study of the variety of symplectic Lie algebras has promising uses in homogeneous geometry and representation theory. In addition to those already mentioned, the existence of \textit{open orbits} or \textit{closed orbits} in $\mathcal{L}_{\omega_{\tiny{\mbox{cn}}}}(\mathbb{R}^{2n})$ allows us to define a kind of distinguished symplectic structures on Lie algebras. For instance, it follows easily from results obtained by Albert Nijenhuis and Roger Richardson in \cite[\S 24]{NijenhuisRichardson}, that the orbits $\operatorname{Sp}(4, \mathbb{R}) \cdot \mu_{18}$ and $\operatorname{Sp}(4, \mathbb{R}) \cdot \mu_{19}$, which correspond to the symplectic structures $\omega_{2}$ and $\omega_{3}$ on $\mathfrak{d}_{4,2}$, are open subsets of $\mathcal{L}_{\omega_{\tiny{\mbox{cn}}}}(\mathbb{R}^{4})$.

About closed orbits in the variety of symplectic Lie algebras, let us mention that there exist nontrivial such orbits in $\mathcal{L}_{\omega_{\tiny{\mbox{cn}}}}(\mathbb{R}^{6})$. Let $\tau$ be the Lie algebra law
$$
\tau =
\left\{
\tau({\it e_1},{\it e_3}) =  {\it e_3},\,
\tau({\it e_1},{\it e_6}) = -{\it e_6},\,
\tau({\it e_2},{\it e_4}) =  {\it e_5},\,
\tau({\it e_4},{\it e_5}) =  {\it e_2}.
\right.
$$
The Lie algebra $(\mathbb{R}^{6},\tau)$ is isomorphic to $\mathfrak{rr}_{3,0}\times \mathfrak{rr}'_{3,0}$ and
by using results of Patrick Eberlein and Michael Jablonski \cite{EberleinJablonski}, an easy computation shows that the $\operatorname{Sp}(6, \mathbb{R})$-orbit of $\tau$ is a closed subset of $\mathcal{L}_{\omega_{\tiny{\mbox{cn}}}}(\mathbb{R}^{6})$.
Independently of the geometric applications of closed orbits of $\mathcal{L}_{\omega_{\tiny{\mbox{cn}}}}(\mathbb{R}^{2n})$, recall that such sets are \textit{distinguished}; from a invariant-theoric point of view.



\newgeometry{margin=1.5cm}
\section*{Appendix}
Here we list essential degenerations to obtain \textbf{Theorem A.}

\begin{itemize}
  \item $(\mathfrak{d}_{4,2},\omega_{2}) \longrightarrow (\mathfrak{r}_{4,-\frac{1}{2},\frac{1}{2}},\omega)$: $g_{t}\cdot \mu_{ 18} \xrightarrow[{t \to +\infty}]{}\mu_{12}(-\frac{1}{2})$ with $g_{t} = \operatorname{diag}\left( 1,{{\rm e}^{\frac{1}{2},t}} , 1 ,{{\rm e}^{-\frac{1}{2},t}}\right)$.

  \item $(\mathfrak{d}_{4,2},\omega_{3}) \longrightarrow (\mathfrak{r}_{4,-\frac{1}{2},\frac{1}{2}},\omega)$: $g_{t}\cdot \mu_{19 } \xrightarrow[{t \to +\infty}]{} \mu_{12}(-\frac{1}{2})$ with $g_{t} = \operatorname{diag}\left( 1,{{\rm e}^{\frac{1}{2},t}} , 1 ,{{\rm e}^{-\frac{1}{2},t}}\right)$.

  \item $(\mathfrak{d}_{4,2},\omega_{3}) \longrightarrow (\mathfrak{d}_{4,2},\omega_{1})$: $g_{t}\cdot \mu_{19 } \xrightarrow[{t \to +\infty}]{} \mu_{17}$ with $g_{t} = \left[ \begin {array}{cccc} 0&1&0&0\\ \noalign{\medskip}0&0&-{{\rm e}
^{-t}}&0\\ \noalign{\medskip}0&0&1&1\\ \noalign{\medskip}{{\rm e}^{t}}
&-{{\rm e}^{t}}&0&0\end {array} \right]$.

   \item $(\mathfrak{r}_{2}\mathfrak{r}_{2},\omega_{\lambda}) \longrightarrow (\mathfrak{d}_{4,1},\omega_{1})$: $g_{t}\cdot \mu_{5 }(\lambda) \xrightarrow[{t \to +\infty}]{}\mu_{15}$ with $g_{t} =  \left[ \begin {array}{cccc} 0&1&0&0\\ \noalign{\medskip}0&0&{{\rm e}^
{-t}}&0\\ \noalign{\medskip}0&0&1&1\\ \noalign{\medskip}-{{\rm e}^{t}}
&{{\rm e}^{t}}&0&0\end {array} \right].$

  \item $(\mathfrak{r}_{2}\mathfrak{r}_{2},\omega_{\lambda}) \longrightarrow (\mathfrak{rr}_{3,0},\omega)$: $g_{t}\cdot \mu_{5 } \xrightarrow[{t \to +\infty}]{} \mu_{2}$ with $g_{t} = \operatorname{diag}(1,{{\rm e}^{t}},1,{{\rm e}^{-t}})$.

  \item     $(\mathfrak{r}{'}_{2},\omega) \longrightarrow (\mathfrak{d}_{4,1},\omega_{1})$: $g_{t}\cdot \mu_{6 } \xrightarrow[{t \to +\infty}]{} \mu_{15}$ with $g_{t} = \left[ \begin {array}{cccc} 1&0&0&0\\ \noalign{\medskip}0&0&0&{
{\rm e}^{\frac{1}{2}\,t}}\\ \noalign{\medskip}0&0&1&0\\ \noalign{\medskip}0&-{
{\rm e}^{-\frac{1}{2}\,t}}&0&0\end {array} \right].$

\item $(\mathfrak{d}_{4,1},\omega_{2}) \longrightarrow (\mathfrak{d}_{4,1},\omega_{1})$: $g_{t}\cdot \mu_{16 } \xrightarrow[{t \to +\infty}]{} \mu_{15}$ with $g_{t} = \operatorname{diag}(1,{{\rm e}^{-t}},1,{{\rm e}^{t}})$.

\item $(\mathfrak{d}_{4,1},\omega_{2}) \longrightarrow (\mathfrak{n}_{4},\omega)$: $g_{t}\cdot \mu_{ 16} \xrightarrow[{t \to +\infty}]{} \mu_{7}$ with $g_{t} = \left[ \begin {array}{cccc} {{\rm e}^{t}}&0&0&0\\ \noalign{\medskip}0
&{{\rm e}^{2\,t}}&0&0\\ \noalign{\medskip}0&{{\rm e}^{4\,t}}&{{\rm e}^
{-t}}&0\\ \noalign{\medskip}{{\rm e}^{3\,t}}&{{\rm e}^{3\,t}}&0&{
{\rm e}^{-2\,t}}\end {array} \right].$

\item $(\mathfrak{d}_{4,1},\omega_{1}) \longrightarrow (\mathfrak{rh}_{3},\omega)$: $g_{t}\cdot \mu_{15 } \xrightarrow[{t \to +\infty}]{} \mu_{1}$ with $g_{t} = \left[ \begin {array}{cccc} -{{\rm e}^{t}}&0&0&0\\ \noalign{\medskip}0
&1&0&0\\ \noalign{\medskip}0&-{{\rm e}^{t}}&-{{\rm e}^{-t}}&0
\\ \noalign{\medskip}{{\rm e}^{2\,t}}&0&0&1\end {array} \right].$

\item $(\mathfrak{rr}_{3,0},\omega) \longrightarrow (\mathfrak{rh}_{3},\omega) $: $g_{t}\cdot \mu_{ 2} \xrightarrow[{t \to +\infty}]{} \mu_{1}$ with $g_{t} = \left[ \begin {array}{cccc} {{\rm e}^{t}}&0&0&0\\ \noalign{\medskip}0
&1&0&0\\ \noalign{\medskip}0&-{{\rm e}^{t}}&{{\rm e}^{-t}}&0
\\ \noalign{\medskip}-{{\rm e}^{2\,t}}&0&0&1\end {array} \right].$

\item $(\mathfrak{h}_{4},\omega_{+}) \longrightarrow  (\mathfrak{d}_{4,\frac{1}{2}} , \omega)$: $g_{t}\cdot \mu_{23 } \xrightarrow[{t \to +\infty}]{} \mu_{20}(\frac{1}{2})$ with $g_{t} = \operatorname{diag}(1,{{\rm e}^{-\frac{1}{2}\,t}},1,{{\rm e}^{\frac{1}{2}\,t}})$.

\item $(\mathfrak{h}_{4},\omega_{+}) \longrightarrow (\mathfrak{n}_{4},\omega)$: $g_{t}\cdot \mu_{ 23} \xrightarrow[{t \to +\infty}]{} \mu_{7}$ with $g_{t} = \left[ \begin {array}{cccc} -\frac{1}{4}\,{{\rm e}^{2\,t}}&0&0&0\\
     \noalign{\medskip}0&{{\rm e}^{-t}}&0&0\\
     \noalign{\medskip}0&-\frac{1}{4}\,{{\rm e}^{3\,t}}&-4\,{{\rm e}^{-2\,t}}&0\\
     \noalign{\medskip}\frac{1}{16}\,{{\rm e}^{6\,t}}&\frac{1}{2}\,{{\rm e}^{t}}&0&{{\rm e}^{t}}
     \end {array}
 \right].
$

\item $(\mathfrak{h}_{4},\omega_{-}) \longrightarrow  (\mathfrak{d}_{4,\frac{1}{2}} , \omega)$: $g_{t}\cdot \mu_{24 } \xrightarrow[{t \to +\infty}]{} \mu_{20}(\frac{1}{2})$ with $g_{t} = \operatorname{diag}(1,{{\rm e}^{-\frac{1}{2}\,t}},1,{{\rm e}^{\frac{1}{2}\,t}})$.

\item $(\mathfrak{h}_{4},\omega_{-}) \longrightarrow (\mathfrak{n}_{4},\omega)$: $g_{t}\cdot \mu_{24 } \xrightarrow[{t \to +\infty}]{} \mu_{7}$ with $g_{t} = \left[ \begin {array}{cccc}
\frac{1}{4}\,{{\rm e}^{2\,t}}&0&0&0\\
 \noalign{\medskip}0&{{\rm e}^{-t}}&0&0\\
 \noalign{\medskip}0&-\frac{1}{4}\,{{\rm e}^{3\,t}}&4\,{{\rm e}^{-2\,t}}&0\\
 \noalign{\medskip}-\frac{1}{16}\,{{\rm e}^{6\,t}}&-\frac{1}{2}\,{{\rm e}^{t}}&0&{{\rm e}^{t}}
 \end {array}
 \right].$

\item $(\mathfrak{r}_{4,0},\omega_{+}) \longrightarrow (\mathfrak{rr}_{3,0},\omega)$: $g_{t}\cdot \mu_{8 } \xrightarrow[{t \to +\infty}]{} \mu_{2}$ with $g_{t} = \operatorname{diag}(1,{{\rm e}^{-\frac{1}{2}\,t}},1,{{\rm e}^{\frac{1}{2}\,t}})$.

\item $(\mathfrak{r}_{4,0},\omega_{+}) \longrightarrow (\mathfrak{n}_{4},\omega)$: $g_{t}\cdot \mu_{ 8} \xrightarrow[{t \to +\infty}]{} \mu_{7}$ with $g_{t} = \left[ \begin {array}{cccc} -{{\rm e}^{2\,t}}&0&0&0
\\ \noalign{\medskip}0&{{\rm e}^{-t}}&0&0\\ \noalign{\medskip}0&-{
{\rm e}^{3\,t}}&-{{\rm e}^{-2\,t}}&0\\ \noalign{\medskip}{{\rm e}^{6\,
t}}&-{{\rm e}^{t}}&0&{{\rm e}^{t}}\end {array} \right]
.$

\item $(\mathfrak{r}_{4,0},\omega_{-}) \longrightarrow (\mathfrak{rr}_{3,0},\omega)$: $g_{t}\cdot \mu_{9 } \xrightarrow[{t \to +\infty}]{} \mu_{2}$ with $g_{t} = \operatorname{diag}(1,{{\rm e}^{-\frac{1}{2}\,t}},1,{{\rm e}^{\frac{1}{2}\,t}}).$

\item $(\mathfrak{r}_{4,0},\omega_{-}) \longrightarrow (\mathfrak{n}_{4},\omega)$: $g_{t}\cdot \mu_{ 9} \xrightarrow[{t \to +\infty}]{} \mu_{7}$ with $g_{t} =
\left[ \begin {array}{cccc}
{{\rm e}^{2\,t}}&0&0&0\\
\noalign{\medskip}0&{{\rm e}^{-t}}&0&0\\
\noalign{\medskip}0&-{{\rm e}^{3\,t}}&{{\rm e}^{-2\,t}}&0\\
\noalign{\medskip}-{{\rm e}^{6\,t}}&{{\rm e}^{t}}&0&{{\rm e}^{t}}
\end {array} \right].$

\item $(\mathfrak{d}_{4,2} , \omega_{1}) \longrightarrow  (\mathfrak{n}_{4},\omega)$: $g_{t}\cdot \mu_{ 17} \xrightarrow[{t \to +\infty}]{} \mu_{7}$ with $g_{t} =
\left[
\begin {array}{cccc}
0&-1&0&0\\
\noalign{\medskip}{{\rm e}^{2\,t}}&{{\rm e}^{t}}&0&0\\
\noalign{\medskip}0&\frac{1}{6}\,{{\rm e}^{3\,t}}&{{\rm e}^{-t}}&-1\\
\noalign{\medskip}\frac{1}{2}\,{{\rm e}^{3\,t}}&\frac{1}{2}\,{{\rm e}^{2\,t}}&{{\rm e}^{-2\,t}}&0
\end {array} \right]
.$

\item $(\mathfrak{r}_{4,-1},\omega) \longrightarrow  (\mathfrak{r}_{4,-1,-1},\omega)$: $g_{t}\cdot \mu_{ 10} \xrightarrow[{t \to +\infty}]{} \mu_{11}(-1)$ with $g_{t} =
\left[ \begin {array}{cccc}
1&0&0&0\\
\noalign{\medskip}0&0&0&-{{\rm e}^{t}}\\
\noalign{\medskip}0&0&1&0\\
\noalign{\medskip}0&{{\rm e}^{-t}}&0&0
\end {array} \right]
$

\item $(\mathfrak{r}_{4,-1},\omega) \longrightarrow (\mathfrak{n}_{4},\omega)$: $g_{t}\cdot \mu_{ 10} \xrightarrow[{t \to +\infty}]{} \mu_{7}$ with $g_{t} = \left[ \begin {array}{cccc}
-2\,{{\rm e}^{t}}&0&0&0\\
\noalign{\medskip}0&-4\,{{\rm e}^{2\,t}}&0&0\\
\noalign{\medskip}0&-4\,{{\rm e}^{4\,t}}&-\frac{1}{2}\,{{\rm e}^{-t}}&0\\
\noalign{\medskip}-2\,{{\rm e}^{3\,t}}&4\,{{\rm e}^{3\,t}}&0&-\frac{1}{4}\,{{\rm e}^{-2\,t}}
\end {array} \right].
$

\item $(\mathfrak{d}_{4,\frac{1}{2}},\omega) \longrightarrow (\mathfrak{rh}_{3},\omega)$: $g_{t}\cdot \mu_{20}(\frac{1}{2}) \xrightarrow[{t \to +\infty}]{} \mu_{1}$ with $g_{t} =  \left[ \begin {array}{cccc} {{\rm e}^{-t}}&0&0&0\\ \noalign{\medskip}0
&1&0&0\\ \noalign{\medskip}0&-2\,{{\rm e}^{2\,t}}&{{\rm e}^{t}}&0
\\ \noalign{\medskip}-2\,{{\rm e}^{t}}&0&0&1\end {array} \right].$

\item $(\mathfrak{r}_{4,-1,-1},\omega) \longrightarrow  (\mathfrak{rh}_{3},\omega)$: $g_{t}\cdot \mu_{11 }(-1) \xrightarrow[{t \to +\infty}]{} \mu_{1}$ with $g_{t} =
\left[ \begin {array}{cccc}
0&1&0&0\\
\noalign{\medskip}-2\,{{\rm e}^{t}}&0&0&0\\
\noalign{\medskip}0&-{{\rm e}^{t}}&0&1\\
\noalign{\medskip}0&0&-\frac{1}{2}\,{{\rm e}^{-t}}&0
\end {array} \right].$

\item $(\mathfrak{d}_{4,\lambda},\omega) \longrightarrow (\mathfrak{n}_{4},\omega)$ with $\lambda \neq \frac{1}{2},1,2$: $g_{t}(\lambda)\cdot \mu_{ 20}(\lambda) \xrightarrow[{t \to +\infty}]{} \mu_{7}$ with $g_{t}(\lambda) =
\left[ \begin {array}{cccc}
0&-1&0&0\\
\noalign{\medskip}{{\rm e}^{2\,t}} \left( -1+\lambda \right) &{{\rm e}^{t}}&0&0\\
\noalign{\medskip}-{\frac {{{\rm e}^{4\,t}} \left( -1+\lambda \right) ^{2}}{\lambda}}&{\frac { \left( -1+\lambda \right) ^{2}{
{\rm e}^{3\,t}}}{\lambda\, \left( -1+2\,\lambda \right) }}&{\frac {{
{\rm e}^{-t}}}{-1+\lambda}}&-1\\
\noalign{\medskip}0&{\frac {{{\rm e}^{2\,t}} \left( -1+\lambda \right) }{\lambda}}&{\frac {{{\rm e}^{-2\,t}
}}{-1+\lambda}}&0
\end {array} \right].$

\item $(\mathfrak{d}{'}_{4,\delta}, \omega_{\pm}) \longrightarrow (\mathfrak{n}_{4},\omega)$:

$g_{t}(\pm) \cdot \mu_{21(+)|22(-)}(\delta) \xrightarrow[{t \to +\infty}]{} \mu_{7}$ with $g_{t}(\pm) =
\left[ \begin {array}{cccc}
-\frac{1}{4}\,{{\rm e}^{\frac{1}{2}\,t}} \left( 4+{\delta}^{2} \right) &0&0&0\\
\noalign{\medskip}0&{{\rm e}^{-\frac{1}{4}\,t}}&0&0 \\
\noalign{\medskip}0&-\frac{1}{4}\,{{\rm e}^{\frac{3}{4}\,t}} \left( 4+{\delta}^{2} \right) & -4\,{\frac {{{\rm e}^{-\frac{1}{2}\,t}}}{4+{\delta}^{2}}}&0\\ \noalign{\medskip}\frac{1}{16}\,{{\rm e}^{\frac{3}{2}\,t}} \left( 4+{\delta}^{2} \right) ^{2}& \pm \frac{1}{2} \,  {{\rm e}^{\frac{1}{4}\,t}}\delta&0&{{\rm e}^{\frac{1}{4}\,t}}
\end {array} \right].$

\item $(\mathfrak{r}_{4,-1,\beta},\omega)  \longrightarrow (\mathfrak{n}_{4},\omega)$ with $\beta \neq -1$:
$g_{t}(\beta) \cdot \mu_{11 }(\beta) \xrightarrow[{t \to +\infty}]{} \mu_{7}$ with $g_{t}(\beta) =
\left[ \begin {array}{cccc}
0&-1&0&0\\
\noalign{\medskip}{{\rm e}^{2\,t}} \left( -1+\beta \right) &{{\rm e}^{t}}&0&0\\
\noalign{\medskip}\frac{1}{2}\,{\frac {{{\rm e}^{4\,t}} \left( -1+\beta \right) ^{3}}{\beta+1}}&\frac{1}{2}\,{\frac {{{\rm e}^{3\,t}} \left( -1+\beta \right) ^{2}}{\beta+1}}&{\frac {{{\rm e}^{-t}}}{-1+\beta}}&-1\\
\noalign{\medskip}\frac{1}{2}\,{\frac {{{\rm e}^{3\,t}} \left( -1+\beta \right) ^{2} \left( \beta-3 \right) }{\beta+1}}&-{\frac {{{\rm e}^{2\,t}} \left( -1+\beta \right) }{\beta+1}}&{\frac {{{\rm e}^{-2\,t}}}{-1+\beta}}&0
\end {array} \right].$

\item $(\mathfrak{r}_{4,\alpha,-\alpha},\omega) \longrightarrow (\mathfrak{n}_{4},\omega)$: $g_{t}(\alpha)\cdot \mu_{12}(\alpha) \xrightarrow[{t \to +\infty}]{} \mu_{7}$ with $g_{t}(\alpha) =
\left[ \begin {array}{cccc}
0&-1&0&0\\
\noalign{\medskip}-{\frac {{{\rm e}^{2\,t}} \left( \alpha+1 \right) }{\alpha}}&{{\rm e}^{t}}&0&0\\
\noalign{\medskip}-\frac{1}{2}\,{\frac {{{\rm e}^{4\,t}} \left( \alpha+1 \right) ^{3}}{{\alpha}^{2} \left( -1+\alpha \right) }}&\frac{1}{2}\,{\frac {{
{\rm e}^{3\,t}} \left( \alpha+1 \right) ^{2}}{\alpha\, \left( -1+\alpha \right) }}&-{\frac {{{\rm e}^{-t}}\alpha}{\alpha+1}}&-1\\ \noalign{\medskip}-\frac{1}{2}\,{\frac {{{\rm e}^{3\,t}} \left( \alpha+1 \right) ^{2} \left( 1+3\,\alpha \right) }{{\alpha}^{2} \left( -1+
\alpha \right) }}&{\frac {{{\rm e}^{2\,t}} \left( \alpha+1 \right) }{-1+\alpha}}&-{\frac {{{\rm e}^{-2\,t}}\alpha}{\alpha+1}}&0
\end {array}
 \right].$

\item $(\mathfrak{r}{'}_{4,0,\delta},\omega_{\pm}) \longrightarrow (\mathfrak{n}_{4},\omega)$: $g_{t}(\pm,\delta)\cdot \mu_{13(+)|14(-)} \xrightarrow[{t \to +\infty}]{} \mu_{7}$ with $g_{t}(\pm,\delta) = \left[
\begin {array}{cccc}
{\frac {\mp{{\rm e}^{\frac{1}{2}\,t}} \left( {\delta}^{2}+1 \right) }{\delta}}&0&0&0\\
\noalign{\medskip}0&{{\rm e}^{-\frac{1}{4}\,t}}&0&0\\
\noalign{\medskip}0&-{\frac {{{\rm e}^{\frac{3}{4}\,t}} \left( {\delta}^{2}+1 \right) }{{\delta}^{2}}}&
{\frac {\mp{{\rm e}^{-\frac{1}{2}\,t}}\delta}{{\delta}^{2}+1}}&0\\
\noalign{\medskip}{\frac {\pm{{\rm e}^{\frac{3}{2}\,t}} \left( {\delta}^{2}+1 \right) ^{2}}{{\delta}^{3}}}&{\frac {\mp{{\rm e}^{\frac{1}{4}\,t}}}{\delta}}&0&{{\rm e}^{\frac{1}{4}\,t}}
\end {array} \right].
$

\item $(\mathfrak{rr}_{3,-1},\omega)\longrightarrow (\mathfrak{n}_{4},\omega) $: $g_{t}\cdot \mu_{3 } \xrightarrow[{t \to +\infty}]{} \mu_{7}$ with $g_{t} = \left[ \begin
{array}{cccc} 0&-1&0&0\\
\noalign{\medskip}-{{\rm e}^{2\,t}}&{{\rm e}^{t}}&0&0\\
\noalign{\medskip}0&\frac{1}{2}\,{{\rm e}^{3\,t}}&-{{\rm e}^{-t}}&-1\\
\noalign{\medskip}-{{\rm e}^{3\,t}}&{{\rm e}^{2\,t}}&-{{\rm e}^{-2\,t}}&0\end {array} \right].$

\item $(\mathfrak{rr}{'}_{3,0},\omega) \longrightarrow (\mathfrak{n}_{4},\omega) $: $g_{t}\cdot \mu_{4 } \xrightarrow[{t \to +\infty}]{} \mu_{7}$ with $g_{t} =
\left[ \begin {array}{cccc}
{{\rm e}^{\frac{1}{2}\,t}}&0&0&0 \\
\noalign{\medskip}0&{{\rm e}^{-\frac{1}{4}\,t}}&0&0\\
\noalign{\medskip}0&-{{\rm e}^{\frac{3}{4}\,t}}&{{\rm e}^{-\frac{1}{2}\,t}}&0\\
\noalign{\medskip}-{{\rm e}^{\frac{3}{2}\,t}}&0&0&{{\rm e}^{\frac{1}{4}\,t}}
\end {array} \right].$

 \item $(\mathfrak{n}_{4},\omega) \longrightarrow  (\mathfrak{rh}_{3},\omega)$: $g_{t}\cdot \mu_{7 } \xrightarrow[{t \to +\infty}]{1} \mu_{}$ with $g_{t} =  \left[ \begin {array}{cccc} {{\rm e}^{t}}&0&0&0\\ \noalign{\medskip}0
&{{\rm e}^{t}}&0&0\\ \noalign{\medskip}0&0&{{\rm e}^{-t}}&0
\\ \noalign{\medskip}0&-{{\rm e}^{2\,t}}&0&{{\rm e}^{-t}}\end {array}
 \right].$

 \item $(\mathfrak{rh}_{3},\omega) \longrightarrow (\mathfrak{a}_{4},\omega)$: $g_{t}\cdot \mu_{1 } \xrightarrow[{t \to +\infty}]{} \mu_{0}$ with $g_{t} = \operatorname{diag}(1,{{\rm e}^{t}},1,{{\rm e}^{-t}}).$
\end{itemize}

\end{document}